\documentclass[11pt, twoside]{article}

\usepackage{graphicx}
\usepackage[caption=false]{subfig}
\usepackage{amsmath}
\usepackage{amssymb,amsfonts}
\usepackage{amsthm}
\usepackage{bm}
\usepackage{mathrsfs}
\usepackage{amssymb}
\usepackage{multirow}
\usepackage{stmaryrd}

\usepackage[margin=1in]{geometry}

\usepackage{algorithm}
\usepackage{algorithmic}

\numberwithin{equation}{section}
\usepackage{multirow}
\usepackage{makecell}
\usepackage{booktabs}
\theoremstyle{definition}
\newtheorem{theorem}{Theorem}

\newtheorem{lemma}{Lemma}

\newtheorem{definition}{Definition}
\newtheorem{remark}{Remark}

\usepackage{cite}
\usepackage{hyperref}
\usepackage[nameinlink]{cleveref}

\newcommand{\h}{\mathfrak{h}}

\usepackage{fancyhdr}
\begin{document}
	
	\title{Optimal $L^2$ error estimation for the unfitted interface finite element method based on the  non-symmetric Nitsche's methods}

\author{Gang Chen%
	\thanks{ College of  Mathematics, Sichuan University, Chengdu 610064, China (\mbox{cglwdm@scu.edu.cn}).}
	\and
	Chaoran Liu%
	\thanks{ College of  Mathematics, Sichuan University, Chengdu 610064, China.}
	\and
	Yangwen Zhang%
	\thanks{Department of Mathematics, University of Louisiana at Lafayette, Lafayette, LA, USA (\mbox{yangwen.zhang@louisiana.edu}).}
}

\date{\today}

\maketitle

\begin{abstract}
	This paper establishes optimal error estimates in the $L^2$ for the non-symmetric Nitsche method in an unfitted interface finite element setting. Extending our earlier work, we give a complete analysis for the Poisson interface model and, by formulating a tailored dual problem that restores adjoint consistency, derive the desired bounds.
\end{abstract}

\section{Introduction.}

Let $\Omega$ be a convex polygonal/polyhedron domain in $\mathbb R^d$ ($d=2,3$) with an immersed interface $\Gamma$.
The interface $\Gamma$ is assumed as enclosed and $C^2$. The domian is divedid into two domains $\Omega_1$ and $\Omega_2$. We consider the following model problem: Find $u\in H^1(\Omega_1)\cap H^1(\Omega_2)$ such that
\begin{align}\label{org}
	\begin{split}
		-\nabla\cdot(\mu\nabla u)&=f\qquad\ \ \text{in } \Omega,\\
		[\![\mu\nabla u\cdot\bm n]\!]&= g_N\qquad \text{on }\Gamma, \\
		[\![u]\!]&=g_D\qquad\text{on }\Gamma,\\
		u&=0\qquad\text{on }\partial\Omega,
	\end{split}
\end{align}
where $\mu=\mu_i$ on $\Omega_i$ ($i=1,2$) are piece-wise constants, $u=u_i$ on $\Omega_i$ ($i=1,2$), $f\in L^2(\Omega)$, $g_N\in H^{\frac 1 2}(\Gamma)$, $g_D\in H^{\frac 3 2}(\Gamma)$, and $\bm n$ is the unit normal vector of $\Gamma$ from $\Omega_1$ point to $\Omega_2$, and 
\begin{align*}
	[\![u]\!] = u_1-u_2,\qquad
	[\![\mu\nabla u\cdot\bm n]\!]
	=\mu_1\nabla u_1\cdot\bm n {-}
	\mu_2\nabla u_2\cdot\bm n.
\end{align*}
{ 
	Interface problems find extensive applications in engineering and scientific fields, serving as a fundamental and ubiquitous subject of study in areas such as multiphysics coupling, multiphase flow, and composite material mechanics. The PDEs describing these problems are coupled across an interface, which partitions the domain into several subdomains. These subdomains may exhibit complex geometries or possess distinct material properties, for instance, consider the transmission conditions for temperature and heat flux across an interface separating media with different thermal conductivities in heat conduction problems.
	
	For the numerical solution of interface problems, careful consideration must be given to both the geometric complexity of the interface and the strict enforcement of interface conditions. Traditional fitted finite element methods often incur significant computational costs when handling complex or evolving interfaces due to their requirement for strict conformity between the mesh and the interface geometry.
	Consequently, unfitted finite element methods have emerged as a well-established approach for dealing with interface problems. Leveraging their defining characteristic that the mesh need not conform to the interface, these methods employ strategies such as cutting mesh elements across the interface and implementing specialized interface condition treatments. Within the unfitted FEM framework, a popular numerical technique utilizes the concept originally proposed by Nitsche \cite{MR341903} for weakly enforcing Dirichlet boundary conditions via penalty. The Nitsche's method, as a flexible technique for weak imposition, incorporates interface penalty terms into the variational formulation. This approach simultaneously circumvents the limitations of body-fitted meshes and, crucially, guarantees theoretical equivalence to the original problem formulation. As a result, the Nitsche's method is widely adopted for the simulation of interface problems using unfitted finite elements.
	
	The seminal work originates from Hansbo and Hansbo  \cite{MR1941489} in 2002, who introduced a two-dimensional unfitted finite element method based on the symmetric Nitsche's formulation. Their scheme's stability critically depends on a sufficiently large stabilization parameter. Subsequently, in 2018 (arxiv version is in 2016), Burman \cite{MR3800035} developed a stabilized unfitted finite element method within the symmetric Nitsche framework, where an optimal L²-norm error estimate was established.
	Motivated by the goal of eliminating the stringent requirement on the stabilization parameter, in 2017, \cite{MR3668542} introduced an non-symmetric Nitsche formulation. While the sufficiently large requirement on the penalty parameter is omitted, its theoretical analysis yields only a suboptimal $L^2$ error estimate despite numerical evidence suggesting that it \emph{may} be optimal. Boiveau \cite{boiveau2015fitted} introduced a penalty-free non-symmetric Nitshce's method to apply unfitted problem, which also obtained a sub-optimal $L^2$-norm estimate theoretically but reached optimal in numerical results. 
	
	The origin of this theoretical deficiency is essentially linked to the specific characteristics of interface problems based on non-symmetric Nitsche's method. The introduction of non-symmetric terms further breaks the adjoint-consistency of the variational formulation, rendering the classical duality argument (Aubin-Nitsche trick) used for $L^2$-norm estimates difficult to apply directly. Crucially, the interaction between the non-symmetric penalty term and these error sources has not been fully revealed. Therefore, breaking through the theoretical barrier preventing optimal $L^2$-norm estimates for the non-symmetric Nitsche's method in unfitted simulations of interface problems is important. It is to the benefit of strengthening the theoretical foundations of unfitted finite elements and advancing their application to complex interface problems in engineering.
	
	This paper focuses on improving the $L^2$ error estimation for the unfitted finite element method based on the non-symmetric Nitsche's method. The main purposes include:
	\begin{itemize}
		\item Constructing an improved duality argument framework suitable for non-symmetric variational formulations. This framework will be combined with specialized regularity analysis near the interface to overcome the theoretical obstacles posed by non-symmetry.
		
		\item Rigorously proving optimal $L^2$-norm error estimates for the non-symmetric Nitsche method applied to classical interface problems \eqref{org} and validating the theoretical results through numerical experiments.
	\end{itemize}
}




The outline of this paper is as follows: In \Cref{section2}, we give some key notations and basic results. In \Cref{section3}, we introduce the non-symmetric Nitshce's methods briefly. In \Cref{section4}, we establish the regularity  analysis of dual problem we will use in $L^2$ norm estimate. In \Cref{section5}, we construct a special dual problem to prove the optimal $L^2$ estimation.

\section{Notations and Basic results}\label{section2}
{ In this section, we present some notations and basic  results to be used in this paper, as well as the non-symmetric Nitsche's method for the problem \eqref{org}.
	
	\subsection{Notations}
	We first introduce some key notations. 
	For $s \ge0$, let $H^{s}(\Omega)$, $H_{0}^{s}(\Omega)$ and $H^s(\Gamma)$ represent standard s-th-order Sobolev spaces on $\Omega$ and $\Gamma$, with $L^2(\Omega)=H^0(\Omega)$. Their dual spaces are $H^{-s}(\Omega)$ and $H^{-s}(\Gamma)$, respectively. Norms $\|\cdot\|_{s, \Omega}$, $\|\cdot\|_{s, \Gamma}$ and semi-norms $|\cdot|_{s, \Omega}$, $|\cdot|_{s, \Gamma}$ act on these spaces. For simplicity, we often use $\|\cdot\|_s$ and $|\cdot|_s $ to replace $\|\cdot\|_{s,\Omega}$ and $|\cdot|_{s,\Omega}$. 
	We denote the $L^2$-inner product on $L^2(\Omega)$ by
	\begin{align*}
		(u,v)=\int_\Omega uv\;{\rm d}x.
	\end{align*}
	In particular, when $\Gamma\subset \mathbb{R}^{d-1}$, we use inner product on $L^2(\Gamma)$: $\langle\cdot, \cdot\rangle_{\Gamma}=\int_\Gamma uv\;{\rm d}\sigma$ to replace $(\cdot, \cdot)_{\Gamma}$. 
}

We define the space 
\begin{align*}
	V=\{v\in L^2(\Omega): v|_{\Omega_1}\in H^1(\Omega_1),v|_{\Omega_2}\in H^1(\Omega_2), v|_{\partial\Omega}=0\}.
\end{align*}
For any $v\in V$, we also use the notation
\begin{align}
	v_1:=v|_{\Omega_1}\qquad v_2:=v|_{\Omega_2}.
\end{align}

We define the weighted averages of $v\in V$ on the interface $\Gamma$ as
\begin{align}
	\{\!\!\{v\}\!\!\}
	=w_1v_1+w_2v_2.
\end{align}
where the non-negative constants $w_1, w_2$ satisfies $w_1+w_2=1$, where $w_1$ and $w_2$ will be specified later.

We define the jump of $v\in V$ on the interface $\Gamma$ as
\begin{align}
	[\![v]\!] = v_1-v_2.
\end{align}

{ 
	\subsection{Basic results}
	We now state basic results used later in the article.
	
	It's easy to get the following lemma with the help of \cite[Lemma 12.15]{Ern2021} and Young's inequality.
	\begin{lemma}[The trace inequality]\label{trace}
		For any $T\in\mathcal T_h$ and $v_h\in \mathcal P_k(T)$, there holds
		\begin{align*}
			\|v_h\|_{{\partial T}}\le C(h_T^{-\frac 1 2} \|v_h\|_{T}+h_{{T}}^{\frac 1 2}\|\nabla v_h\|_{{T}}),
		\end{align*}
		{where $\mathcal{P}_k(T)$ is the polynomial space on the element $T$.}
	\end{lemma}

	Here and in what follows, $C$ denotes a generic positive constant independent of both the mesh parameter $h,h_T,h_E,\mathfrak{h}$ and the parameter $c_0, \mu_1,\mu_2 $.

	The following lemma is a direct corollary of \cite[Theorem 4.5.11]{Brenner2008}.
	\begin{lemma}[The inverse inequality]\label{inverse}
		Let $T\in\mathcal{T}_h$ be  quasi-uniform, and $\mathcal{P}_k(T)$ be the polynomial space on the element $T$. Then, for any $v_h \in \mathcal{P}_k(T)$ and $T \in \mathcal{T}_h$, we have
		\[
		|v_h|_{1,T} \leq C h_T^{-1}\|v_h\|_{0,T}.
		\]
	\end{lemma}

}

\section{Finite element method}\label{section3}

{ 
	Let $\mathcal{T}_{h}=\cup\{T\}$ be a shape regular (there exits a constant $\sigma>0$ such that for all $T\in\mathcal{T}_h$, $h_T/\rho_T\le\sigma$ where $\rho_T$ is the radius of the inscribed ball in $T$)  simplicial decomposition of the domain $\Omega$. Here, $h=\max_{T \in \mathcal{T}_{h}} h_{T}$, and $h_{T}$ is the diameter of $T$. We decouple $\mathcal{T}_h$ into two parts according to the division of region $\Omega$. They are respectively  denoted as 
	\begin{align}
		\mathcal T_h^1= \{ T\in\mathcal T_h, T\cap \Omega_1\neq \emptyset  \},\qquad	\mathcal T_h^2= \{ T\in\mathcal T_h, T\cap \Omega_2\neq \emptyset  \}.
	\end{align}
	
	Then set 
	\begin{align}
		\mathcal T_h^{\Gamma}
		=\{T\in\mathcal T_h, \overline T\cap\Gamma\neq \emptyset \}.
	\end{align}
	$\mathcal{T}_h^{\Gamma}$ is the set of all elements that intersect the interface, where $\overline{T}$ denotes the closure of $T$.
	Let $\mathcal{E}_{h}=\cup \{E\}$ be the union of all the edges (faces) of $T \in \mathcal{T}_{h}$. Denote $h_E$ as the diameter of $E$, and define $\mathfrak{h}$ such that $\mathfrak{h}|_T:=h_T$ and $\mathfrak{h}|_E:=h_E$.
	
	Then we define the finite element space as follow:}


\begin{align}
	\begin{split}
		V_h=&\{v_h\in L^2(\Omega):
		v_h|_{\Omega_1}=v_{h,1},
		v_h|_{\Omega_2}=v_{h,2},v_{h,1}\in H^1(\mathcal T_h^1)\cap \mathcal P_k(\mathcal T_h^1),
		\\
		&
		v_{h,2}\in H^1(\mathcal T_h^2)\cap \mathcal P_k(\mathcal T_h^2), v_1|_{\partial\Omega_1/\Gamma}=0,
		v_2|_{\partial\Omega_2/\Gamma}=0\}.
	\end{split}
\end{align}

{ 
	For any $T\in\mathcal T_h$, we define the standard $L^2$-projection $\Pi_{k}^o: L^2(T)\to \mathcal P_k(T)$ as:
	\begin{align}\label{pro1}
		(\Pi_k^ou, v_h)_T&=(u, v_h)_T \quad \forall v_h\in  \mathcal P_k(T).
	\end{align}
	This kind of projection will be applied in \Cref{section5} for estimating the standard Scott-Zhang interpolation.
	For the projections $\Pi_k^o$ with $k\ge0$, the following stability and approximation results are basic from \cite[Theorem 2.6]{MR2431403} and inverse inequality.

	\begin{lemma}For any $T\in\mathcal{T}_h$ and nonnegative integer $k$, it holds
		\begin{align}
			&  \| u-\Pi_k^o u \|_{0, T}+h_T | u-\Pi_k^o u |_{1, T}\lesssim  h_T^{s}|v|_{s, T}, \quad\forall u \in H^{s}(T), \\
			&  \|\Pi_k^o v \|_{0, T} \leq\|v\|_{0, T}, \quad\forall v \in L^2(T),
		\end{align}
		where $1 \leq s \leq k+1$.
	\end{lemma}
}

\begin{lemma} \label{lem:por-averge}We have the following equations
	\begin{align*}
		\{\!\!\{\mu\nabla u\cdot\bm n\}\!\!\}
		&=\mu_1\nabla u_1\cdot\bm n-w_2	[\![\mu\nabla u\cdot\bm n]\!],\\
		\{\!\!\{\mu\nabla u\cdot\bm n\}\!\!\}
		&=\mu_2\nabla u_2\cdot\bm n+w_1	[\![\mu\nabla u\cdot\bm n]\!]
	\end{align*}
\end{lemma}
\begin{proof} By a direct calculate, one has
	\begin{align*}
		\{\!\!\{\mu\nabla u\cdot\bm n\}\!\!\}
		&=w_1 \mu_1\nabla u_1\cdot\bm n+w_2\mu_2\nabla u_2\cdot\bm n\\
		&=\mu_1\nabla u_1\cdot\bm n+ (w_1-1) \mu_1\nabla u_1\cdot\bm n+w_2\mu_2\nabla u_2\cdot\bm n\\
		&=\mu_1\nabla u_1\cdot\bm n-w_1	[\![\mu\nabla u\cdot\bm n]\!].
	\end{align*}
	Similarity, it holds
	\begin{align*}
		\{\!\!\{\mu\nabla u\cdot\bm n\}\!\!\}
		&=w_1 \mu_1\nabla u_1\cdot\bm n+w_2\mu_2\nabla u_2\cdot\bm n\\
		&=w_1\mu_1\nabla u_1\cdot\bm n+\mu_2\nabla u_2\cdot\bm n+(w_2-1)w_2\mu_2\nabla u_2\cdot\bm n\\
		&=\mu_1\nabla u_1\cdot\bm n+w_1	[\![\mu\nabla u\cdot\bm n]\!].
	\end{align*}
	
\end{proof}
We followed the idea in \cite{MR2002258,MR2257119,MR2491426,MR3668542},  to choice $w_1, w_2$ and $c_0$ as
\begin{align}
	w_1=\frac{\mu_2}{\mu_1+\mu_2},\qquad 	w_2=\frac{\mu_1}{\mu_1+\mu_2},\qquad 	c_0:=\{\!\!\{\mu\}\!\!\}.
\end{align}
So it holds
\begin{align}\label{por-c0}
	\begin{split}
		&	\min(\mu_1,\mu_2)\le c_0\le  \max(\mu_1,\mu_2),\qquad c_0\le 2\min(\mu_1,\mu_2),\\
		&	 w_1\mu_1\le c_0,\qquad w_2\mu_2\le c_0.
	\end{split}
\end{align}
\begin{remark}
	We address that	\eqref{por-c0} is the key portieres to get the error estimation independent of jump of the coefficient.
	In \cite{MR1941489}, for any $T\in \mathcal T_h^{\Gamma}$, $\Gamma$ separate $T$ into two parts $T_1\subset \Omega_1$ and $T_2\subset \Omega_2$, the authors define
	\begin{align}
		w_1=\frac{|T_1|}{|T|},\qquad 	w_2=\frac{|T_2|}{|T|}.
	\end{align}
	In this case, no matter how to define $c_0$, there is no way to obtain the property \eqref{por-c0}, therefore, there is no way to get the error estimation independent of jump of the coefficient.
\end{remark}

We define 
\begin{align}
	\begin{split}
		a(u, v)
		&=\mu_1(\nabla u_1, \nabla v_1)_{\Omega_1}
		+ \mu_2(\nabla u_2, \nabla v_2)_{\Omega_2}\\
		&\quad
		-\langle \{\!\!\{\mu\nabla u\cdot\bm n\}\!\!\},[\![v]\!] \rangle_{\Gamma}
		+\langle [\![u]\!],\{\!\!\{\mu\nabla v\cdot\bm n\}\!\!\} \rangle_{\Gamma}
		+\langle c_0\h^{-1}[\![u]\!],[\![v]\!] \rangle_{\Gamma},
	\end{split}
\end{align}
and
\begin{align}
	\ell(v)=(f, v)+\langle  g_N, w_2v_1+w_1v_2\rangle_{\Gamma}
	+\langle
	g_D,\{\!\!\{\mu\nabla v\cdot\bm n\}\!\!\}
	\rangle_{\Gamma}
	+\langle c_0\h^{-1}g_D,[\![v]\!] \rangle_{\Gamma}.
\end{align}
Then by \Cref{lem:por-averge} we know that, the  solution of \eqref{org} are also the solution of the follow problem:  Find $u\in V$ such that
\begin{align}\label{fem-org}
	a(u, v)= \ell(v)
\end{align}
holds for all $v\in V$.

The finite element method reads: Find $u_h\in V_h$ such that
\begin{align} \label{fem}
	a(u_h, v_h)=\ell(v_h).
\end{align}

For any $v\in V$, we define the norm
\begin{align}
	\interleave v\interleave^2=\mu_1\|\nabla v_1\|_{\Omega_1}^2
	+\mu_2\|\nabla v_2\|_{\Omega_2}^2
	+c_0\|\h^{-\frac 1 2}[\![v]\!]\|_{\Gamma}^2.
\end{align}
Then  it is obviously that for any $v_h\in V_h$, it holds
\begin{align*}
	a(v_h,v_h)=	\interleave v_h\interleave^2.
\end{align*}
and by the definition of $a(\cdot,\cdot)$, the inverse inequality, for any $u_h, v_h\in V_h$, it holds
{ 
	\begin{align*}
		|a(u_h, v_h)|&\le C(
		\sum_{i=1}^2\mu_i^{\frac 1 2}\|\nabla u_{h, i}\|_{\Omega_i}
		+c_0^{-\frac{1}{2}}\|\h^{\frac 1 2}\{\!\!\{\mu\nabla u_h\cdot\bm n\}\!\!\}\|_{\Gamma}
		+
		c_0^{\frac 1 2}\|\h^{-\frac 1 2}[\![u_h]\!]\|_{\Gamma}
		) \\
		&\quad
		(
		\sum_{i=1}^2\mu_i^{\frac 1 2}\|\nabla v_{h, i}\|_{\Omega_i} 
		+c_0^{-\frac{1}{2}}\|\h^{\frac 1 2}\{\!\!\{\mu\nabla v_h\cdot\bm n\}\!\!\}\|_{\Gamma}
		+
		c_0^{\frac 1 2}\|\h^{-\frac 1 2}[\![v_h]\!]\|_{\Gamma}
		) \\
		&\le C(
		\sum_{i=1}^2\mu_i^{\frac 1 2}\|\nabla u_{h, i}\|_{\Omega_i} 
		+\sum_{i=1}^2(w_i\mu_i)^{\frac 1 2}\|\nabla u_{h, i}\|_{\Omega_i}
		+
		c_0^{\frac 1 2}\|\h^{-\frac 1 2}[\![u_h]\!]\|_{\Gamma}
		) \\
		&\quad
		(
		\sum_{i=1}^2\mu_i^{\frac 1 2}\|\nabla v_{h, i}\|_{\Omega_i} 
		+\sum_{i=1}^2(w_i\mu_i)^{\frac 1 2}\|\nabla v_{h, i}\|_{\Omega_i}
		+
		c_0^{\frac 1 2}\|\h^{-\frac 1 2}[\![v_h]\!]\|_{\Gamma}
		) \\
		&\le C	\interleave u_h\interleave \cdot	\interleave v_h\interleave.
\end{align*}}
With the above coerivetiy and continuity we have the following result  {by Lax-Milgram Theorem}.
\begin{lemma}[Existence of a unique solution] The problem \eqref{fem} has a unique solution $v_h\in V_h$.
\end{lemma}

{ From \Cref{fem-org} and \Cref{fem}, we have the following lemma.}

\begin{lemma}[Orthogonality]\label{orth} Let $u\in V$ and $u_h\in V_h$ be the solution of \eqref{org} and \eqref{fem}, respectively. Then we have
	\begin{align}
		a(u-u_h, v_h)=0\qquad \forall v_h\in V_h
	\end{align}
	holds for all $v_h\in V$.
\end{lemma}

\section{Some novel regularity results}\label{section4}

\subsection{The basic regularity results}

The following theorems are basic  {and} useful in our proof.
\begin{theorem} \label{cylinder-1}
	Let $\Lambda$ be a bounded convex domain,
	$\mathcal O$ be a bounded domain inside $\Lambda$ with $C^2$ boundary  $\partial\mathcal O$. Let $f\in L^2(\Lambda/\mathcal O)$, $g_{out}\in H^{\frac 3 2}(\partial\Lambda)$, $g_{in}\in H^{\frac 3 2}(\partial\mathcal O)$,  find $u\in H^1(\Lambda/\mathcal O)$ such that
	\begin{align*}
		-\Delta u&=f \qquad\quad \text{in }\Lambda/\mathcal O,\\
		u&=g_{in}\qquad\; \text{on }\partial\mathcal O,\\
		u&=g_{out} \qquad \text{on }\partial\Lambda.
	\end{align*}
	Then the above problem has a unique solution $u\in H^1(\Lambda/\mathcal O)$ and
	\begin{align*}
		\|u\|_{2,\Lambda/\mathcal O}\le C(\|f\|_{\Lambda/\mathcal O}
		+\|g_{out}\|_{\frac 3 2, \partial\Lambda}
		+\|g_{in}\|_{\frac 3 2,\partial\mathcal O}
		).
	\end{align*}
\end{theorem}

\begin{theorem} \label{cylinder-2}
	Let $\Lambda$ be a bounded convex domain,
	$\mathcal O$ be a bounded domain inside $\Lambda$ with $C^2$ boundary  $\partial\mathcal O$. Let $f\in L^2(\Lambda/\mathcal O)$, $g_{out}\in H^{\frac 3 2}(\partial\Lambda)$, $g_{in}\in H^{\frac 1 2}(\partial\mathcal O)$,  find $u\in H^1(\Lambda/\mathcal O)$ such that
	\begin{align*}
		-\Delta u&=f \qquad \text{in }\Lambda/\mathcal O,\\
		\nabla u\cdot\bm n&=g_{in}\qquad \text{on }\partial\mathcal O,\\
		u&=g_{out} \qquad \text{on }\partial\Lambda,
	\end{align*}
	Then the above problem has a unique solution $u\in H^1(\Lambda/\mathcal O)$ and
	\begin{align*}
		\|u\|_{2,\Lambda/\mathcal O}\le C(\|f\|_{\Lambda/\mathcal O}
		+\|g_{out}\|_{\frac 3 2, \partial\Lambda}
		+\|g_{in}\|_{\frac 1 2,\partial\mathcal O}
		).
	\end{align*}
\end{theorem}

\begin{lemma}\label{nuoyiman} Suppose that $\Omega$ is a bounded convex domain, $\Delta u\in L^2(\Omega)$, $u\in H^1(\Omega)$ and $\nabla u\cdot\bm n\in H^{\frac 1 2}(\Gamma)$,  then we have the estimation 
	\begin{align}\label{neumann-bound}
		\|u\|_2\le C
		(\|\Delta u\|_0+\|\nabla u\cdot\bm n\|_{\frac 1 2,\Gamma})+\|u\|_0.
	\end{align}
\end{lemma}

\subsection{The regularity for interface problems}
\begin{theorem}[{\cite[Theorem 4.4]{MR2257456}}] \label{basic00}
	For $f\in L^2(\Omega)$ and $g\in H^{\frac 1 2}(\Gamma)$,
	we consider the elliptic problems
	with  robin interface condition: Find $u\in V$ such that
	\begin{align}\label{basic-def00}
		\begin{split}
			-\nabla\cdot(\mu\nabla u)&=f\qquad\  \text{in }\Omega,\\
			[\![u]\!]&=0\qquad \text{on }\Gamma,\\
			[\![\mu\nabla u\cdot\bm n]\!]&=g\qquad \text{on }\Gamma.
		\end{split}
	\end{align}
	Then the above problem has a unique solution $u\in V$.
	In addition, when $\Omega$ is convex, we have the regularity estimation
	\begin{align}\label{def-C00}
		\mu_1\| u\|_{2,\Omega_1} + 	\mu_2\| u\|_{2,\Omega_2}\le C(\|f\|_0
		+\|g\|_{\frac 1 2,\Gamma}
		),
	\end{align}
	where $C$ is independent of $\mu$, but dependent on $\Omega$.
\end{theorem}

In next theorem, we are going to extend the above regularity result to $[\![u]\!]\neq 0$. 
\begin{theorem}\label{basic0}
	For $f\in L^2(\Omega)$, $g_N\in H^{\frac 1 2}(\Gamma)$ and $g_D\in H^{\frac 3 2}(\Gamma)$,
	we consider the elliptic problems
	with  robin interface condition: 
	Find $u\in V$ such that
	\begin{align}\label{basic-def0}
		\begin{split}
			-\nabla\cdot(\mu\nabla u)&=f\qquad\  \text{in }\Omega,\\
			[\![u]\!]&=g_D\qquad \text{on }\Gamma,\\
			[\![\mu\nabla u\cdot\bm n]\!]&=g_N\qquad \text{on }\Gamma.
		\end{split}
	\end{align}
	Then the above problem has a unique solution $u\in V$.
	In addition, when $\Omega$ is convex, we have the regularity estimation
	\begin{align}\label{def-C0}
		\mu_1\| u\|_{2,\Omega_1} + 	\mu_2\| u\|_{2,\Omega_2}\le C(\|f\|_0+\min(\mu_1,\mu_2)\|g_D\|_{\frac 3 2,\Gamma}
		+\|g_N\|_{\frac 1 2,\Gamma}
		),
	\end{align}
	where $C$ is independent of $\mu$, but dependent on $\Omega$.
\end{theorem}
\begin{proof}
	We define the problem: Find $w\in H^1(\Omega_1)$ such that
	\begin{align*}
		-\Delta w &= 0\qquad \text{in }\Omega,\\
		w &= g_D\qquad\!\!\!\! \text{on }\Gamma,\\
		w &= 0 \qquad\text{on }\partial\Omega_1/\Gamma.
	\end{align*}
	Then by \eqref{cylinder-1}, we have the regularity estimation
	$$\|w\|_2\le C\|g_D\|_{\frac 3 2,\Gamma}.$$ 
	Then we set 
	\begin{align}
		\widetilde u_1= u_1-w,\qquad \widetilde u_2=u_2.
	\end{align}
	So \eqref{basic-def0} can be rewrite as
	\begin{align}
		\begin{split}
			-\nabla\cdot( \mu_1\nabla\widetilde u_1)&= f-\nabla\cdot(\mu_1\nabla w)\qquad \text{in }\Omega_1,\\
			-\nabla\cdot( \mu_2\nabla\widetilde u_2)&= f\qquad\qquad\qquad\qquad\  \text{in }\Omega_2,\\
			[\![\widetilde u]\!]&=0,\qquad\qquad \qquad \qquad  \text{on }\Gamma,\\
			[\![\mu\nabla \widetilde u\cdot\bm n]\!]&=g_N {-}\mu_1\nabla w\cdot\bm n,\qquad \text{on }\Gamma.
		\end{split}
	\end{align}
	Then we use \Cref{basic00} to get
	\begin{align*}
		\mu_1\|\widetilde u_1\|_{2,\Omega_1}
		+\mu_2\|\widetilde u_2\|_{2,\Omega_2}
		&\le C( 
		\|f-\nabla\cdot(\mu_1\nabla w)\|_{\Omega_1}
		+\|f\|_{\Omega_2}
		+\|g_N {-}\mu_1\nabla w\cdot\bm n\|_{\frac 1 2, \Gamma}
		)\\
		&\le C\|f\|_0+\mu_1\|w\|_{2,\Omega_1}+\|g_N\|_{\frac 1 2,\Gamma}\\
		&\le C\|f\|_0+\mu_1\|g_D\|_{\frac 3 2,\Gamma}+\|g_N\|_{\frac 1 2,\Gamma},
	\end{align*}
	which leads to
	\begin{align}
		\mu_1\| u\|_{2,\Omega_1} + 	\mu_2\| u\|_{2,\Omega_2}\le C(\|f\|_0+\mu_1\|g_D\|_{\frac 3 2,\Gamma}
		+\|g_N\|_{\frac 1 2,\Gamma}
		),
	\end{align}
	Similarity, we can get
	\begin{align}
		\mu_1\| u\|_{2,\Omega_1} + 	\mu_2\| u\|_{2,\Omega_2}\le C(\|f\|_0+\mu_2\|g_D\|_{\frac 3 2,\Gamma}
		+\|g_N\|_{\frac 1 2,\Gamma}
		),
	\end{align}
	which finishes our proof.
\end{proof}

\subsection{Regularity for Robin interface condition}

{In this subsection, we provide a frame to analyze the regularity for Robin interface condition and study the influence of parameters on regularity under this boundary condition.}

{ First, we present the regularity analysis under Robin boundary conditions when the jump term is not affected by parameters.}

\begin{theorem}\label{basic-robin}
	For $f\in L^2(\Omega)$, $g_N\in H^{\frac 1 2}(\Gamma)$ and $g_D\in H^{\frac 1 2}(\Gamma)$, we consider the elliptic problem
	with  robin interface condition: Find $u\in V$ such that
	\begin{align}\label{basic-def}
		\begin{split}
			-\nabla\cdot(\mu\nabla u)&=f\qquad\ \text{in }\Omega,\\
			[\![u]\!]+\mu_1\nabla u_1\cdot\bm n&=g_D\qquad \text{on }\Gamma,\\
			[\![\mu\nabla u\cdot\bm n]\!]&=g_N\qquad \text{on }\Gamma.
		\end{split}
	\end{align}
	Then the above problem has a unique solution $u\in V$.
	In addition, when $\Omega$ is convex, we have the regularity estimation
	\begin{subequations}\label{def-C}
		\begin{align}
			\mu_1\| u\|_{2,\Omega_1} \le C\max\{\mu_1^{-1},\mu_2^{-1},1\}(\|f\|_0+\|g_D\|_{\frac 1 2,\Gamma}
			+\|g_N\|_{\Gamma}
			),
		\end{align}
		and
		\begin{align}
			\mu_2\| u\|_{2,\Omega_2}\le C\max\{\mu_1^{-1},\mu_2^{-1},1\}(\|f\|_0+\|g_D\|_{\frac 1 2,\Gamma}
			+\|g_N\|_{\frac 1 2,\Gamma}
			),
		\end{align}
	\end{subequations}
	where $C$ is independent of $\mu$, but dependent on $\Omega$.
\end{theorem}
\begin{proof}
	If $\partial\Omega_1/\Gamma=\emptyset$, then $\partial\Omega_2/\Gamma\neq\emptyset$, and by the Poincaré inequality  to get
	{ 
		\begin{align}
			\|u_2\|_{\Omega_2}\le C \|u_2\|_{1, \Omega_2}\le  C\|\nabla u_2\|_{\Omega_2},
		\end{align}
	}
	and 
	\begin{align*}
		{ \|u_1\|_{\Omega_1}\le C \|u_1\|_{1,\Omega_1}}\le& C(\|\nabla u_1\|_{\Omega_1} +\langle u_1, 1 \rangle_{\Gamma})\\
		=&  C(\|\nabla u_1\|_{\Omega_1} +\langle u_1-u_2, 1 \rangle_{\Gamma}+\langle u_2, 1 \rangle_{\Gamma})\\
		\le& C(\|\nabla u_1\|_{\Omega_1}+
		\|[\![u]\!]\|_{\Gamma}
		+\|u_2\|_{\Gamma}) \\
		\le & C(\|\nabla u_1\|_{\Omega_1}+
		\|[\![u]\!]\|_{\Gamma}
		+\|u_2\|_{1, \Omega_2}) \\
		\le & C(\|\nabla u_1\|_{\Omega_1}
		+\|\nabla u_2\|_{\Omega_2}+
		\|[\![u]\!]\|_{\Gamma}).
	\end{align*}
	Similarity,  if  {$\partial\Omega_1/\Gamma\neq\emptyset$ and  $\partial\Omega_2/\Gamma=\emptyset$}, we have
	{ 
		\begin{align}
			\|u_1\|_{\Omega_1}\le C\|u_1\|_{1,\Omega_1}\le C\|\nabla u_1\|_{\Omega_1},
		\end{align}
	}
	and
	\begin{align*}
		{ \|u_2\|_{\Omega_2}\le \|u_2\|_{\Omega_2}}\le
		C(\|\nabla u_1\|_{\Omega_1}
		+\|\nabla u_2\|_{\Omega_2}+
		\|[\![u]\!]\|_{\Gamma}).
	\end{align*}
	Either way, we have 
	\begin{align}
		\|u_1\|_{\Omega_1}+\|u_2\|_{\Omega_2}
		&\le{  \|u_1\|_{1,\Omega_1}+\|u_2\|_{1,\Omega_2} }\\
		&\le C(\|\nabla u_1\|_{\Omega_1}
		+\|\nabla u_2\|_{\Omega_2}+
		\|[\![u]\!]\|_{\Gamma})
	\end{align}
	Therefore,
	\begin{align}
		\begin{split}
			\|u_1\|_{\frac 1 2, \Gamma}
			+\|u_2\|_{\frac 1 2, \Gamma}&\le C
			(\|u_1\|_{1,\Omega_1}
			+
			\|u_2\|_{1, \Omega_2})\\
			&\le
			C(\|\nabla u_1\|_{\Omega_1}
			+\|\nabla u_2\|_{\Omega_2}+
			\|[\![u]\!]\|_{\Gamma}).
		\end{split}
	\end{align}
	
	We take $v\in V$ to get the varitonal form
	\begin{align*}
		\mu_1(\nabla u_1,\nabla v_1)_{\Omega_1}
		+\mu_2(\nabla u_2,\nabla v_2)_{\Omega_2}
		+\langle[\![u]\!], [\![v]\!]\rangle_{\Gamma}
		=(f, v)+\langle g_D,  { [\![v]\!]}\rangle_{\Gamma}
		+\langle g_N,  v_2\rangle_{\Gamma}
		.
	\end{align*}
	We set $v=u$ in the above equation to get
	\begin{align*}
		&\mu_1\|\nabla u_1\|_{\Omega_1}^2
		+\mu_2\|\nabla u_2\|_{\Omega_2}^2
		+\|[\![u]\!]\|_{\Gamma}^2\\
		&\qquad \le 
		\|f\|_{\Omega_1}\| u_1\|_{\Omega_1}
		+ \|f\|_{\Omega_2}\|u_2\|_{\Omega_2}
		+\|g_D\|_{\Gamma}\| {[\![u]\!]}\|_\Gamma+\|g_N\|_{\Gamma}\| {u_2}\|_\Gamma\\
		&\qquad\le C(\|\nabla u_1\|_{\Omega_1}
		+\|\nabla u_2\|_{\Omega_2}+
		\|[\![u]\!]\|_{\Gamma})(\|f\|_0 +\|g_D\|_{\Gamma}+\|g_N\|_{\Gamma})\\
		&\qquad
		\le C(\mu_1^{-1}+\mu_2^{-1}+1)(\|f\|_0^2 +\|g_D\|^2_{\Gamma}+\|g_N\|_{\Gamma}^2)\\
		&\qquad\quad
		+\frac 12 (
		\mu_1\|\nabla u_1\|_{\Omega_1}^2
		+\mu_2\|\nabla u_2\|_{\Omega_2}^2
		+\|[\![u]\!]\|_{\Gamma}^2
		),
	\end{align*}
	which leads to
	\begin{align*}
		& \mu_1\|\nabla u_1\|_{\Omega_1}^2
		+\mu_2\|\nabla u_2\|_{\Omega_2}^2
		+\|[\![u]\!]\|_{\Gamma}^2
		\le C {(\mu_1^{-1}+\mu_2^{-1}+1)}(\|f\|_0^2 +\|g_D\|^2_{\Gamma}+\|g_N\|_{\Gamma}^2),
	\end{align*}
	Therefore,
	\begin{align*}
		& \|[\![u]\!]\|_{\frac 1 2,\Gamma}\le C(\|u_1\|_{1,\Omega_1}+ {\|u_2\|_{1,\Omega_2}})
		\le  C { \max\{\mu_1^{-1},\mu_2^{-1},1\}(\|f\|_0 +\|g_D\|_{\Gamma}+\|g_N\|_{\Gamma})}.
	\end{align*}
	By \eqref{basic-def}, we can get
	\begin{align*}
		-\nabla\cdot(\mu_1\nabla u_1) &= f\qquad\qquad \quad\text{in }\Omega_1,\\
		\mu_1\nabla u_1\cdot\bm n&= g_D-[\![u]\!]\qquad \text{on }\Gamma,\\
		u_1&=0\qquad\qquad\quad \text{on }\partial\Omega_1/\Gamma.
	\end{align*}
	If $\partial\Omega_1/\Gamma=\emptyset$, by \Cref{nuoyiman}, it holds
	\begin{align*}
		\mu_1\|u_1\|_2&\le C(
		\|f\|_{\Omega_1}
		+
		\|g_D-[\![u]\!]\|_{\frac 1 2,\Gamma}
		+\|u\|_{\Omega_1}
		)\\
		&
		\le C\max\{\mu_1^{-1},\mu_2^{-1},1\}(\|f\|_0+\|g_D\|_{\frac 1 2, \Gamma}+\|g_N\|_{\Gamma}). 
	\end{align*}
	If $\partial\Omega_1/\Gamma\neq \emptyset$, by \Cref{cylinder-2}, it holds
	\begin{align*}
		\mu_1\|u_1\|_2&\le C(
		\|f\|_{\Omega_1}
		+
		\|g_D-[\![u]\!]\|_{\frac 1 2,\Gamma}
		)\\
		&
		\le C {\max\{\mu_1^{-1},\mu_2^{-1},1\}}(\|f\|_0+\|g_D\|_{\frac 1 2, \Gamma}+\|g_N\|_{\Gamma}). 
	\end{align*}
	Either way, it holds
	\begin{align*}
		\mu_1\|u_1\|_2
		&
		\le C\max\{\mu_1^{-1},\mu_2^{-1},1\}(\|f\|_0+\|g_D\|_{\frac 1 2, \Gamma}+\|g_N\|_{\Gamma}). 
	\end{align*}
	Again, let $\bm n_2=-\bm n$, by \eqref{basic-def}, we can get
	\begin{align*}
		-\nabla(\mu_2\nabla u_2)&=f\qquad\qquad\qquad\qquad\text{in }\Omega_2\\
		\mu_2\nabla u_2\cdot\bm n_2&=g_N-\mu_1\nabla u_1\cdot\bm n
		\qquad \!\text{on }\Gamma,\\
		u_2&=0\qquad\qquad\qquad\qquad \text{on }\partial\Omega_2/\Gamma.
	\end{align*}
	Similarity, by \Cref{nuoyiman} or \Cref{cylinder-2}, it holds
	\begin{align*}
		\mu_2\|u_2\|_{\Omega_2}&\le C
		(
		\|f\|_{\Omega_2}+
		\|g_N-\mu_1\nabla u_1\cdot\bm n\|_{\frac 1 2,\Gamma} 
		+\|u_2\|_{\Omega_2}
		)\\
		&\le C(\|f\|_{\Omega_2}+\|g_N\|_{\frac 1 2,\Gamma}+\mu_1\|u_1\|_2 
		+\|u_2\|_{\Omega_2}
		)\\
		&\le   C\max\{\mu_1^{-1},\mu_2^{-1},1\}(\|f\|_0
		+ \|g_N\|_{\frac 1 2,\Gamma}
		+ \|g_D\|_{\frac 1 2, \Gamma}).
	\end{align*}
	
\end{proof}
\begin{remark}
	The regularity for \eqref{basic-robin} was established in \cite{MR3328193} with $g_D=g_N=0$. However, the regularity constant is highly dependent on $\mu$. In the above theorem, we have proved the regularity result with the constant independent of $\mu$. 
\end{remark}

\begin{remark}
	If we change the boundary condition 
	$$[\![u]\!]+\mu_1\nabla u_1\cdot\bm n=g_D$$ into
	$$\gamma[\![u]\!]+\mu_1\nabla u_1\cdot\bm n=g_D$$
	with a constant $\gamma>0$, the constant in \eqref{def-C} \emph{may} dependent on $\gamma$. Particularly, we are interesting in the case that $\gamma\to\infty$. The answer is that $C$ is independent of $\gamma$ when $\gamma\to \infty$.
	We will illustrate this in the following of this subsection.
\end{remark}

\begin{theorem} \label{robin-reg1}
	For $f\in L^2(\Omega)$ and $g\in H^{\frac 1 2}(\Gamma)$, we consider the following problem: Find $u\in V$ such that
	\begin{align}\label{def-robin-1}
		\begin{split}
			-\nabla\cdot(\mu\nabla u)&=0\qquad\ \text{in }\Omega,\\
			\gamma[\![u]\!]+\mu_1\nabla u_1\cdot\bm n&=g\qquad \text{on }\Gamma,\\
			[\![\mu\nabla u\cdot\bm n]\!]&=0\qquad \text{on }\Gamma.
		\end{split}
	\end{align}
	Then the above problem has a unique solution in $V$. In addition, when $\Omega$ is convex, and $\gamma\ge \gamma_0>0$, we have
	{ 
		\begin{align*}
			&\mu_1\| u\|_{2,\Omega_1} + 	\mu_2\| u\|_{2,\Omega_2}\\&\le  C\max\{\mu_1^{-1},\mu_2^{-1},1\}\left((1+\gamma_0^{-1})(1+\left(\frac{\mu_1}{\mu_2}\right)^{\frac{1}{2}})\|g\|_{\frac 1 2,\Gamma}+ \|g\|_{\frac 1 2,\Gamma}\right),
	\end{align*}}
	where $C$ is independent of $\gamma$, $\gamma_0$ and $\mu$, but dependent on $\Omega$. 
\end{theorem}
\begin{proof}
	For any $v\in V$, we have the variational form for \eqref{def-robin-1}: 
	\begin{align}\label{def-robin-var}
		{ \sum_{i=1}^2(\mu_i\nabla u_i,\nabla v_i)}+\gamma\langle [\![u]\!], [\![v]\!] \rangle_{\Gamma}= \langle g, [\![v]\!] \rangle_{\Gamma}
		.
	\end{align}
	The existence of a unique solution $u\in V$ of \eqref{def-robin-var} is followed by the ‌Lax-Milgram Theorem. We take $v=u$ in the \eqref{def-robin-var}, and use the Cauchy- {Schwarz} inequality, the Young's inequality to get
	\begin{align*}
		\mu_1\|\nabla u_1\|_{\Omega_1}^2
		+\mu_2\|\nabla u_2\|_{\Omega_2}^2
		+\gamma\|[\![u]\!]\|_{\Gamma}^2&\le
		\langle g, [\![u]\!] \rangle_{\Gamma}\le \|g\|_{\Gamma}
		\|[\![u]\!]\|_{\Gamma}\\&{ \le  \frac 1 2\gamma^{-1}\|g\|^2_{\Gamma}
			+\frac 1 2\gamma\|[\![u]\!]\|_{\Gamma}^2},
	\end{align*}
	which leads to
	\begin{align*}
		\mu_1\|\nabla u_1\|_{\Omega_1}^2
		+\mu_2\|\nabla u_2\|_{\Omega_2}^2
		+\gamma\|[\![u]\!]\|_{\Gamma}^2&\le C\gamma^{-1}\|g\|_{\Gamma}^2.
	\end{align*}
	The above inequality implies that
	\begin{align}\label{bound-u-Gamma}
		\|[\![u]\!]\|_{\Gamma}&\le C\gamma^{-1}\|g\|_{\Gamma}.
	\end{align}
	By integration by parts and the  {Cauchy}- {Schwarz} inequality, for any $v\in H^1(\Omega_1)$, we can get
	\begin{align}\label{green}
		\langle \nabla u\cdot\bm n, v\rangle_{\Gamma}
		&=(\Delta u, v)_{\Omega_1} + (\nabla u,\nabla v)_{\Omega_1}
		\le \|\Delta u\|_{\Omega_1}\|v\|_{\Omega_1}+\|\nabla u\|_{\Omega_1}\|v\|_{1, \Omega_1}.
	\end{align}
	We define the problem: 
	For any given $w\in H^{\frac 1 2}(\Gamma)$
	Find $v\in H^1(\Omega_1)$ such that
	\begin{align*}
		-\Delta v&= 0\qquad \text{in }\Omega_1,\\
		v&=w \qquad\text{on }\Gamma,\\
		v&=0\qquad\text{on }\partial\Omega_1/\Gamma,
	\end{align*}
	By \Cref{cylinder-1}, then it holds
	\begin{align}\label{bound-v}
		\|v\|_{1,\Omega_1}\le C\|w\|_{\frac 1 2,\Gamma}.
	\end{align}
	So  by the definition of $H^{-\frac 1 2}$ norm, \eqref{green} and \eqref{bound-v}, it holds
	\begin{align}\label{bound-neg-0.5}
		\begin{split}
			\|\nabla u_1\cdot\bm n\|_{-\frac 1 2,\Gamma}&=\sup_{ w\in H^{\frac 1 2}(\Gamma),\|w\|_{\frac 1 2, \Gamma}=1}\langle\nabla u_1\cdot\bm n, w \rangle_{\Gamma}\\
			&=\sup_{ w\in H^{\frac 1 2}(\Gamma),	\|w\|_{\frac 1 2, \Gamma}=1}\langle\nabla u_1\cdot\bm n, v \rangle_{\Gamma}
			\\
			&\le \sup_{ w\in H^{\frac 1 2}(\Gamma),	\|w\|_{\frac 1 2, \Gamma}=1} \|\Delta u_1\|_{\Omega_1}\|v\|_{\Omega_1}+\|\nabla u_1\|_{\Omega_1}\|v\|_{1, \Omega_1}
			\\
			&\le C( \|\Delta u_1\|_{\Omega_1}+\|\nabla u_1\|_{\Omega_1}  )
			,
		\end{split}
	\end{align}
	With the help of \eqref{bound-neg-0.5}, we will improve the estimation for $\mu_1\|\nabla u_1\|_0^2
	+\mu_2\|\nabla u_2\|_0^2$. By integration by parts, \eqref{def-robin-1}, the  {Cauchy}- {Schwarz} inequality and \eqref{bound-neg-0.5} to get
	{ 
		\begin{align*}
			&\mu_1\|\nabla u_1\|_0^2
			+\mu_2\|\nabla u_2\|_0^2\\
			&\qquad
			=\mu_1(\nabla u_1,\nabla u_1)_{\Omega_1}
			+\mu_2(\nabla u_2,\nabla u_2)_{\Omega_2}\\
			&\qquad
			=-\mu_1(\Delta u_1, u_1)_{\Omega_1}
			+\mu_1\langle\nabla u_1\cdot\bm n,u_1 \rangle_{\Gamma} 
			-\mu_2(\Delta u_2, u_2)_{\Omega_2}
			-\mu_2\langle\nabla u_2\cdot\bm n,u_2 \rangle_{\Gamma} 
			\\
			&\qquad
			=-\mu_1(\Delta u_1, u_1)
			-\mu_2(\Delta u_2, u_2)
			+\mu_1\langle\nabla u_1\cdot\bm n,[\![u]\!] \rangle_{\Gamma}+\langle[\![\mu\nabla u\cdot\bm n]\!], u_2\rangle_\Gamma
			\\
			&\qquad
			=-\mu_1(\Delta u_1, u_1)
			-\mu_2(\Delta u_2, u_2)
			+\mu_1\langle\nabla u_1\cdot\bm n,[\![u]\!] \rangle_{\Gamma} 
			\\
			&\qquad =  \gamma^{-1}\langle\mu_1\nabla u_1\cdot\bm n, g-\mu_1\nabla u_1\cdot\bm n\rangle_{\Gamma} \\
			&\qquad =  \gamma^{-1}\mu_1\langle\nabla u_1\cdot\bm n, g\rangle_{\Gamma} - \gamma^{-1}\mu_1^2\|\nabla u_1\cdot\bm n\|_{\Gamma}^2\\
			&  \qquad \le
			C\gamma^{-1}\mu_1\|g\|_{\frac 1 2,\Gamma}\|\nabla u_1\|_{\Omega_1}.
		\end{align*}
		Then we can conclude
		\begin{align*}
			\|\nabla u_1\|_{0,\Omega_1}\le C\gamma^{-1}\|g\|_{\frac{1}{2},\Gamma},
		\end{align*}
		which leads to
		\begin{align}\label{bound-grand-u}
			\mu_1\|\nabla u_1\|_{0,\Omega_1}^2
			+\mu_2\|\nabla u_2\|_{0,\Omega_2}^2\le C\gamma^{-2}\mu_1\|g\|_{\frac 1 2,\Gamma}^2.
		\end{align}
		We use Poincaré inequality to get
		\begin{align}\label{bound-u-1}
			\|u_1\|_{1,\Omega_1}
			+	\|u_2\|_{1,\Omega_2}&\le C\gamma^{-1}(
			1
			+
			(\frac{\mu_1}{\mu_2})^{\frac 1 2})\|g\|_{\frac 1 2,\Gamma}.
		\end{align}
		At last, we use \Cref{basic-robin}, \eqref{def-robin-1} and the triangle inequality to get
		\begin{align}\label{bound-u-2}
			\begin{split}
				&\mu_1\| u\|_{2,\Omega_1} + 	\mu_2\| u\|_{2,\Omega_2}\\
				&\qquad\le C {\max\{\mu_1^{-1},\mu_2^{-1},1\}}(\|\nabla\cdot(\mu\nabla u)\|_0+\|[\![u]\!]+\mu_1\nabla u_1\cdot\bm n\|_{\frac 1 2,\Gamma} )\\
				&
				\qquad=  C {\max\{\mu_1^{-1},\mu_2^{-1},1\}}\|g-(\gamma-1)[\![u]\!]\|_{\frac 1 2,\Gamma}\\
				&\qquad\le C {\max\{\mu_1^{-1},\mu_2^{-1},1\}}\left(\|g\|_{\frac 1 2,\Gamma}+C(\gamma+1)\|[\![u]\!]\|_{\frac 1 2,\Gamma}\right).
			\end{split}
		\end{align}
		The remain is to bound $(\gamma+1)\|u\|_{\frac 1 2,\Gamma}$. By 
		\eqref{bound-u-1}  to get
		\begin{align}\label{bound-gamma-u}
			\begin{split}
				(1+\gamma)\|[\![u]\!]\|_{\frac12 ,\Gamma} 
				&\le  C(1+\gamma)(\|u\|_{1,\Omega_1}+\|u\|_{1,\Omega_2})  \\
				&\le  C(1+\gamma)\gamma^{-1}(1+\left(\frac{\mu_1}{\mu_2}\right)^{\frac 1 2})\|g\|_{\frac 1 2,\Gamma}\\
				&\le C(1+\gamma_0^{-1})(1+\left(\frac{\mu_1}{\mu_2}\right)^{\frac 1 2})\|g\|_{\frac 1 2,\Gamma}\\
			\end{split}
		\end{align}
		We use \eqref{bound-u-2} and \eqref{bound-gamma-u} to finish our proof.}
\end{proof}

Next, we are ready to prove the following regularity for the Laplace problem with robin boundary condition.
\begin{theorem}[Regularity for robin boundary condition]  \label{robin-reg-2}
	For $f\in L^2(\Omega)$, $g_N\in H^{\frac 1 2}(\Gamma)$ and $g_D\in H^{\frac 1 2}(\Gamma)$, we consider the elliptic problems
	with  robin boundary condition: Find $u\in V$ such that
	\begin{align}\label{def-robin-2}
		\begin{split}
			-\nabla\cdot(\mu\nabla u)&=f\qquad\;\; \text{in }\Omega,\\
			\gamma[\![u]\!]+\mu_1\nabla u_1\cdot\bm n&=g_D\qquad \text{on }\Gamma,\\
			[\![\mu\nabla u\cdot\bm n]\!]&=g_N\qquad \text{on }\Gamma.
		\end{split}
	\end{align}
	Then when $\Omega$ is convex, $\gamma\ge \gamma_0>0$, we have
	{ 
		\begin{align}\label{robin-final-reg}
			&\mu_1\| u\|_{2,\Omega_1} + 	\mu_2\| u\|_{2,\Omega_2}\\
			&\le  C(1+\gamma_0^{-1}) {\max\{\mu_1^{-1},\mu_2^{-1},1\}}(1+\left(\frac{\mu_1}{\mu_2}\right)^\frac{1}{2})(\|f\|_0+\|g_N\|_{\frac 1 2,\Gamma}+\|g_D\|_{\frac 1 2,\Gamma}),
	\end{align}}
	where $C$ is independent of $\gamma$ and $\gamma_0$, but dependent on $\Omega$.
\end{theorem}
\begin{proof} We define the problem: Find $w\in V$ such that
	\begin{align*}
		-\nabla\cdot(\mu\nabla w)&=f\qquad \text{in }\Omega,\\
		[\![w]\!]&=0\qquad \text{on }\Gamma,\\
		[\![\mu\nabla w\cdot\bm n]\!]&=g_N\qquad\!\!\!\! \text{on }\Gamma
	\end{align*}
	Then when $\Omega$ is convex, it holds 
	\begin{align}\label{reg-w}
		\mu_1\| w_1\|_{2,\Omega_1}
		+
		\mu_2\|w_2\|_{2,\Omega_2}
		\le C(\|f\|_0 +\|g_N\|_{\frac 1 2,\Gamma}).
	\end{align} 
	Therefore, we can rewrite the problem as:
	\begin{align*}
		-\nabla\cdot(\mu\nabla (u-w))&=0
		\qquad\qquad\qquad\qquad\; \text{in }\Omega,\\
		\gamma[\![u-w]\!]+\mu_1\nabla (u_1-w_1)\cdot\bm n&=g_D-\mu_1\nabla w_1\cdot\bm n\qquad \text{on }\Gamma,\\
		[\![\mu\nabla (u-w)\cdot\bm n]\!]&=0\qquad\qquad\qquad\qquad\; \text{on }\Gamma
	\end{align*}
	Then by \Cref{robin-reg1}, the triangle inequality and the regularity estimation \eqref{reg-w} we can get
	\begin{align*}
		&\mu_1\|u-w\|_{2,\Omega_1}
		+
		\mu_2\|u-w\|_{2,\Omega_2}\\
		&\quad\le  C {\max\{\mu_1^{-1},\mu_2^{-1},1\}(1+\left(\frac{\mu_1}{\mu_2}\right)^{\frac{1}{2}})}(1+\gamma_0^{-1})\|g_D-\mu_1\nabla w_1\cdot\bm n\|_{\frac 1 2,\Gamma}\\
		&\quad\le  C {\max\{\mu_1^{-1},\mu_2^{-1},1\}(1+\left(\frac{\mu_1}{\mu_2}\right)^{\frac{1}{2}})}(1+\gamma_0^{-1})(\mu_1\|w_1\|_{2,\Omega_1}+\|g_D\|_{\frac 1 2,\Gamma})\\
		&\quad\le  C {\max\{\mu_1^{-1},\mu_2^{-1},1\}(1+\left(\frac{\mu_1}{\mu_2}\right)^{\frac{1}{2}})}(1+\gamma_0^{-1})(\|f\|_0+\|g_N\|_{\frac 1 2,\Gamma}+\|g_D\|_{\frac 1 2,\Gamma}),
	\end{align*}
	which finishes our proof by the triangle inequality and \eqref{reg-w}.
\end{proof}

\subsection{Regularity for the dual problem}\label{regualar-for-dual}

Before, we give the regularity result. We first give two stability results.
\begin{theorem}
	When $\Omega$ is convex, for any given $g\in H^{\frac 1 2}(\Gamma)$, find $u\in V$ such 
	\begin{align*}\label{org-adjoint}
		-\nabla\cdot(\mu \nabla u)&=0\qquad \text{in }\Omega,\\
		[\![\mu \nabla u\cdot\bm n]\!]&=0\qquad \text{on }\Gamma,\\
		[\![u]\!]&=g\qquad \text{on }\Gamma.
	\end{align*}
	Then above problem has a unique solution and it holds the stability
	\begin{align}
		\|u\|_{0}\le  C\|g\|_{-\frac 1 2,\Gamma}.
	\end{align}
\end{theorem}
\begin{proof} For any given $\theta\in L^2(\Omega)$, find $w\in V$ such that
	\begin{align}\label{def-w}
		\begin{split}
			-\nabla\cdot(\mu \nabla w)&=\theta\qquad \text{in }\Omega,\\
			[\![\mu \nabla w\cdot\bm n]\!]&=0\qquad \text{on }\Gamma,\\
			[\![w]\!]&=0\qquad \text{on }\Gamma.
		\end{split}
	\end{align}
	Then it holds the stability 
	\begin{align}\label{reg-ww}
		\mu_1\|w\|_{2,\Omega_1}
		+\mu_2\|w\|_{2,\Omega_2}\le C\|\theta\|_0.
	\end{align}
	We use \eqref{def-w}, integration by parts, the  {Cauchy}- {Schwarz} inequality and the regularity \eqref{reg-ww} to get
	\begin{align*}
		(u,\theta)&=(u, -\nabla\cdot(\mu \nabla w))
		=(\nabla u, \mu\nabla w)
		-\langle g,\mu_1\nabla w_1\cdot\bm n_1 \rangle_{\Gamma}
		\\
		&=-\langle g,\mu_1\nabla w_1\cdot\bm n_1 \rangle_{\Gamma}
		\le \|g\|_{-\frac 1 2,\Gamma}{ \|\mu_1\nabla w_1\cdot\bm n\|_{\frac{1}{2},\Gamma}}\\
		&
		\le C\|g\|_{-\frac 1 2,\Gamma}\mu_1\|w_1\|_{2,\Omega_1}\le  C\|g\|_{-\frac 1 2,\Gamma}\|\theta\|_0
	\end{align*}
	With the above estimation, it arrives
	\begin{align*}
		\|u\|_{0}
		&=\sup_{0\neq \theta\in L^2(\Omega)}
		\frac{(u,\theta)}{\|\theta\|_0}\le C\|g\|_{-\frac 1 2,\Gamma},
	\end{align*}
	which finish our proof.
\end{proof}

\begin{definition} For $u, v\in V$, we define 
	\begin{align}
		\begin{split}
			a^\star(u, v)
			&=\mu_1(\nabla u_1, \nabla v_1)_{\Omega_1}
			+ \mu_2(\nabla u_2, \nabla v_2)_{\Omega_2}\\
			&\quad
			+\langle \{\!\!\{\mu\nabla u\cdot\bm n\}\!\!\},[\![v]\!] \rangle_{\Gamma}
			-\langle [\![u]\!],\{\!\!\{\mu\nabla v\cdot\bm n\}\!\!\} \rangle_{\Gamma}
			+\langle c_0\h^{-1}[\![u]\!],[\![v]\!] \rangle_{\Gamma}.
		\end{split}
	\end{align}
	It is obviously that 
	\begin{align}\label{sys}
		a(v ,u)=a^{\star}(u, v).
	\end{align}
\end{definition}
{ 
	\begin{remark}
		Now let $\psi$ solve the dual problem \eqref{aux}:
		\begin{align*}
			a^{\star}(\psi,v)=(g,v)
		\end{align*}
		for all $v\in V$. Then by \eqref{sys} we can imply 
		\begin{align}\label{adjoint-consis}
			a(v,\psi)=a^{\star}(\psi,v)=(g,v)=(v,g),\quad\forall v\in V.
		\end{align}
		\Cref{adjoint-consis} presents the form \eqref{fem-org} is adjoint consistent defined in \cite{MR1885715}. For the adjoint boundary value problem \eqref{org-adjoint} of the non-symmetric Nitsche's method considered in the past, the corresponding dual problem form is
		\begin{align*}
			a(\psi,v)=(g,v),\quad\forall v\in V,
		\end{align*}
		whose adjoint formulation is 
		\begin{align*}
			a(v,\psi)&=a(\psi,v)-2\langle\{\!\!\{\mu\nabla\psi\cdot\bm n\}\!\!\}, [\![v]\!]\rangle_\Gamma+2\langle[\![\psi]\!],\{\!\!\{\mu\nabla v\cdot\bm n\}\!\!\}\rangle_\Gamma\\
			&=(g,v)-2\langle\{\!\!\{\mu\nabla\psi\cdot\bm n\}\!\!\}, [\![v]\!]\rangle_\Gamma+2\langle[\![\psi]\!],\{\!\!\{\mu\nabla v\cdot\bm n\}\!\!\}\rangle_\Gamma\\
			&=(v,g)-2\langle\{\!\!\{\mu\nabla\psi\cdot\bm n\}\!\!\}, [\![v]\!]\rangle_\Gamma+2\langle[\![\psi]\!],\{\!\!\{\mu\nabla v\cdot\bm n\}\!\!\}\rangle_\Gamma.
		\end{align*}
		Due to non-symmetry, this formulation of adjoint consistency is broken by the additional boundary terms. When applying it to estimate $L^2$-norm, these terms causes the convergence order of its estimation to be approximately $\frac{1}{2}$ lower than the optimal one. The form of the dual problem we designed avoids this boundary term in the adjoint form, thus theoretically achieving the optimal results when dealing with $L^2$-norm estimate.
	\end{remark}
}
Now we can proof the regularity for the dual problem we used in the estimation of $L^2$ norm.
\begin{theorem}[Regularity for the dual problem]\label{dual-reg-1} We define the problem:
	Find $u\in V$ such that
	\begin{align}\label{aux}
		a^\star(u, v)
		=(f, v)
	\end{align}
	holds for all $v\in V$. Then the above problem has a unique solution $u\in V$. In addition, we have
	\begin{align}\label{dual-ref}
		\mu_1\|u\|_{2,\Omega_1}+\mu_2\|u\|_{2,\Omega_2}\le  C{ \max\{\mu_1^{-1},\mu_2^{-1},1\}}\|f\|_0.
	\end{align}
	{ In particular, if $\mu_1,\mu_2\ge1$, we have
		\begin{align}
			\mu_1\|u\|_{2,\Omega_1}+\mu_2\|u\|_{2,\Omega_2}\le  C\|f\|_0.
	\end{align}}
\end{theorem}
\begin{proof}
	The exsitence of a unique solution $u\in V$ followed by the Lax-Milgram Theorem and the coververity
	{ 
		\begin{align}
			a^{\star}(u,u)=\mu_1\|\nabla u_1\|^2_{\Omega_1}
			+\mu_2\|\nabla u_2\|^2_{\Omega_2}+c_0\|\h^{-\frac 1 2}[\![u]\!]\|_{\Gamma}^2.
	\end{align}}
	For the regularity estimation \eqref{dual-ref} we use four steps to finish our proof. We define the  index $i$ as $i=1,2$.
	
	\textbf{Step One, proof of $-\mu_i\Delta u_i = f|_{\Omega_i}$:}  
	
	Since the space $H_0^2(\Omega_i)$ is dense in the space $L^2(\Omega_i)$, then for any $w\in L^2(\Omega_i)$, $\|w\|_{\Omega_i}=1$ and $\varepsilon\in (0, 1)$, there exists a function $v_w\in H_0^2(\Omega_i)$ such that
	\begin{align}
		\|w-v_w\|_{\Omega_i}\le \varepsilon.
	\end{align} 
	We take $v_w\in H_0^2(\Omega_i)$  in \eqref{aux} and use integration by parts to get
	\begin{align}
		-\mu_i(\Delta u_i, v_w)_{\Omega_i}=(f, v_w)_{\Omega_i}.
	\end{align}
	
	So by the definition of $L^2$ norm, it holds 
	\begin{align*}
		\|\mu_i\Delta u_i +f\|_{\Omega_i}&=\sup_{w\in L^2(\Omega_i), \|w\|_{\Omega_i}=1}(\mu_i\Delta u_i +f, w)_{\Omega_i} \\
		&=\sup_{w\in L^2(\Omega_i),\|w\|_{\Omega_i}=1}
		\left[(\mu_i\Delta u_i +f, w-v_w)_{\Omega_i}+(\mu_i\Delta u_i +f, v_w)_{\Omega_i}\right] \\
		&=\sup_{w\in L^2(\Omega_i),\|w\|_{\Omega_i}=1}(\mu_i\Delta u_i {+f}, w-v_w)_{\Omega_i}\\
		&=\sup_{w\in L^2(\Omega_i),\|w\|_{\Omega_i}=1}\|\mu_i\Delta u_i +f\|_{\Omega_i} \|w-v_w\|_{\Omega_i} \\
		&\le \varepsilon\|\mu_i\Delta u_i+f\|_{\Omega_i},
	\end{align*}
	which leads to
	\begin{align*}
		(1-\varepsilon)\| {\mu_i}\Delta u_i +f\|_{\Omega_i}\le 0.
	\end{align*}
	So
	\begin{align*}
		- {\mu_i}\Delta u_i = f|_{\Omega_i}.
	\end{align*}
	\textbf{Step Two, proof of $[\![\mu\nabla u\cdot\bm n]\!]=0$:} 
	
	Since $H^{\frac 1 2}(\Gamma)=H^{\frac 1 2}_0(\Gamma)$ is dense in $L^2(\Gamma)$, then for any $\zeta\in L^2(\Gamma)$, $\|\zeta\|_{\Gamma}=1$ and $0<\varepsilon<1$, there exists $\beta_{\zeta}$
	\begin{align}
		\|\zeta-\beta_{\zeta}\|_{\Gamma}\le \varepsilon.
	\end{align}

	For any $\beta_{\zeta}\in H^{\frac 1 2}(\Gamma)$, we take $v\in V$ such that
	\begin{align}
		\mu_1\nabla v_1\cdot\bm n=\mu_2\nabla v_2\cdot\bm n=0,\qquad 
		v_1|_{\Gamma}=v_2|_{\Gamma}=\beta_{\zeta}.
	\end{align}
	{ 
		So from \eqref{aux} and definition of $a^{\star}(\cdot,\cdot)$, it holds
		\begin{align*}
			0&=\mu_1(\nabla u_1,\nabla v_1)_{\Omega_1}+\mu_2(\nabla u_2,\nabla v_2)_{\Omega_2}-(f,v)\\
			&=-\mu_1(\Delta u_1,v_1)_{\Omega_1}
			-\mu_2(\Delta u_2,v_2)_{\Omega_2}
			-(f, v)+\langle\mu_1\nabla u_1\cdot\bm n,v_1\rangle_\Gamma-\langle\mu_2\nabla u_2\cdot\bm n,v_2\rangle_\Gamma \\
			&=\langle
			[\![\mu\nabla u\cdot\bm n]\!], \beta_\zeta
			\rangle_{\Gamma}-\mu_1(\Delta u_1,v_1)_{\Omega_1}
			-\mu_2(\Delta u_2,v_2)_{\Omega_2}
			-(f, v),
		\end{align*}
		which implies
		\begin{align*}
			\langle
			[\![\mu\nabla u\cdot\bm n]\!], \beta_\zeta
			\rangle_{\Gamma}
			= \mu_1(\Delta u_1,v_1)_{\Omega_1}+
			\mu_2(\Delta u_2,v_2)_{\Omega_2}
			+(f, v) =0.
		\end{align*}
	}
	
	Then it holds
	\begin{align*}
		\|[\![\mu\nabla u\cdot\bm n]\!]\|_{\Gamma}
		&=\sup_{\zeta\in L^2(\Gamma), \|\zeta\|_{\Gamma}=1}
		\langle[\![\mu\nabla u\cdot\bm n]\!], \zeta \rangle_{\Gamma} \\
		&=\sup_{\zeta\in L^2(\Gamma), \|\zeta\|_{\Gamma}=1}\left[
		\langle[\![\mu\nabla u\cdot\bm n]\!], \zeta-\beta_\zeta \rangle_{\Gamma}
		+
		\langle[\![\mu\nabla u\cdot\bm n]\!], \beta_\zeta \rangle_{\Gamma}
		\right]\\
		&=\sup_{\zeta\in L^2(\Gamma), \|\zeta\|_{\Gamma}=1}
		\langle[\![\mu\nabla u\cdot\bm n]\!], \zeta-\beta_\zeta \rangle_{\Gamma}
		\\
		&\le
		\sup_{\zeta\in L^2(\Gamma), \|\zeta\|_{\Gamma}=1}
		\|[\![\mu\nabla u\cdot\bm n]\!]\|_{\Gamma}
		\| \zeta-\beta_\zeta \|_{\Gamma}\\
		&\le
		\varepsilon
		\|[\![\mu\nabla u\cdot\bm n]\!]\|_{\Gamma},
	\end{align*}
	which leads to
	{
		\begin{align*}
			(1-\varepsilon)\|[\![\mu\nabla u\cdot\bm n]\!]\|_{\Gamma}
			\le 0,
	\end{align*}}
	Then 
	\begin{align}
		[\![\mu\nabla u\cdot\bm n]\!]=0.
	\end{align}
	
	\textbf{Step Three, proof of $		\|
		c_0\h^{-1}[\![u]\!]+
		\mu_1\nabla u_1\cdot\bm n\|_{\frac 1 2,\Gamma}\le C\|f\|_{0}$:} 
	
	Since $H^{\frac1 2}(\Gamma)$ is dense in $H^{-\frac 1 2 }(\Gamma)$, then for any $\theta\in H^{-\frac 1 2}(\Gamma), \|\theta\|_{-\frac 1 2,\Gamma}=1$ and $\varepsilon\in (0,1)$, there exists $\beta_{\theta}\in H^{\frac 1 2}(\Gamma)$ such that
	\begin{align}
		\|\theta-\beta_{\theta}\|_{-\frac 1 2,\Gamma}\le \varepsilon.
	\end{align}
	For any given $\beta_{\theta}\in H^{\frac 1 2}(\Gamma)$, find $v\in V$ such 
	\begin{align*}
		-\nabla\cdot(\mu \nabla v)&=0\qquad \text{in }\Omega,\\
		[\![\mu \nabla v\cdot\bm n]\!]&=0\qquad \text{on }\Gamma,\\
		[\![v]\!]&=\beta_{\theta}\qquad\!\! \text{on }\Gamma.
	\end{align*}
	It holds the stability
	\begin{align}
		\|v_{1}\|_{\Omega_1}+\|v_2\|_{\Omega_2}\le  C\|\beta_\theta\|_{-\frac 1 2,\Gamma}.
	\end{align}\label{beta-thtea-reg}
	Then we have
	\begin{align}
		\begin{split}
			a^\star(u, v)
			&=
			\mu_1\langle u_1,\nabla v_1\cdot\bm n \rangle_{\Gamma}
			-\mu_2\langle  u_2, \nabla v_2\cdot\bm n\rangle_{\Gamma}\\
			&\quad	+\langle \{\!\!\{\mu\nabla u\cdot\bm n\}\!\!\},[\![v]\!] \rangle_{\Gamma}
			-\langle [\![u]\!],\{\!\!\{\mu\nabla v\cdot\bm n\}\!\!\} \rangle_{\Gamma}
			+\langle c_0\h^{-1}[\![u]\!],[\![v]\!] \rangle_{\Gamma}\\
			&=\langle \{\!\!\{\mu\nabla u\cdot\bm n\}\!\!\}+c_0\h^{-1}[\![u]\!],\beta_{\theta} \rangle_{\Gamma}.
		\end{split}
	\end{align}
	So it holds
	\begin{align}
		\langle \{\!\!\{\mu\nabla u\cdot\bm n\}\!\!\}+c_0\h^{-1}[\![u]\!],\beta_{\theta} \rangle_{\Gamma}=(f, v_1)_{\Omega_1}+(f, v_2)_{\Omega_2}
	\end{align}
	Then it holds
	{ 
		\begin{align*}
			&\|\{\!\!\{\mu\nabla u\cdot\bm n\}\!\!\}+c_0\h^{-1}[\![u]\!]\|_{\frac 1 2,\Gamma}\\
			&\qquad =\sup_{\theta\in H^{-\frac 1 2}(\Gamma),\|\theta\|_{-\frac 1 2,\Gamma}=1}\langle\{\!\!\{\mu\nabla u\cdot\bm n\}\!\!\}+c_0\h^{-1}[\![u]\!], \theta \rangle_{\Gamma}\\
			&\qquad=\sup_{\theta\in H^{-\frac 1 2}(\Gamma),\|\theta\|_{-\frac 1 2,\Gamma}=1}\langle \{\!\!\{\mu\nabla u\cdot\bm n\}\!\!\}+c_0\h^{-1}[\![u]\!], \theta -\beta_{\theta}+\beta_{\theta}\rangle_{\Gamma}\\
			&\qquad=\sup_{\theta\in H^{-\frac 1 2}(\Gamma),\|\theta\|_{-\frac 1 2,\Gamma}=1}\left[\langle \{\!\!\{\mu\nabla u\cdot\bm n\}\!\!\}+c_0\h^{-1}[\![u]\!], \theta -\beta_{\theta}\rangle_{\Gamma}
			+(f, v_1)_{\Omega_1}+(f, v_2)_{\Omega_2}\right]
			\\
			&\qquad\le 	\|\{\!\!\{\mu\nabla u\cdot\bm n\}\!\!\}+c_0\h^{-1}[\![u]\!]\|_{\frac 1 2,\Gamma}\varepsilon
			+C(\|f\|_{\Omega_1}
			+\|f\|_{\Omega_2})(1+\varepsilon)
			,
	\end{align*}}
	which leads to
	\begin{align}
		\|\{\!\!\{\mu\nabla u\cdot\bm n\}\!\!\}+c_0\h^{-1}[\![u]\!]\|_{\frac 1 2,\Gamma}\le C\frac{1+\varepsilon}{1-\varepsilon}(\|f\|_{\Omega_1}
		+\|f\|_{\Omega_2}).
	\end{align}
	So
	\begin{align}
		\|\{\!\!\{\mu\nabla u\cdot\bm n\}\!\!\}+c_0\h^{-1}[\![u]\!]\|_{\frac 1 2,\Gamma}\le C(\|f\|_{\Omega_1}
		+\|f\|_{\Omega_2}).
	\end{align}
	Therefore,
	\begin{align*}
		\|
		c_0\h^{-1}[\![u]\!]+
		\mu_1\nabla u_1\cdot\bm n\|_{\frac 1 2,\Gamma}
		&=	\|
		c_0\h^{-1}[\![u]\!]+
		\{\!\!\{\mu\nabla u\cdot\bm n\}\!\!\}+ {w_1}[\![\mu\nabla u\cdot\bm n]\!]\|_{\frac 1 2,\Gamma} \\
		&\le\|
		c_0\h^{-1}[\![u]\!]+
		\{\!\!\{\mu\nabla u\cdot\bm n\}\!\!\}\|_{\frac 1 2,\Gamma} 
		+\|[\![\mu\nabla u\cdot\bm n]\!]\|_{\frac 1 2,\Gamma} \\
		&=\|
		c_0\h^{-1}[\![u]\!]+
		\{\!\!\{\mu\nabla u\cdot\bm n\}\!\!\}\|_{\frac 1 2,\Gamma} \\
		&\le  C\|f\|_0.
	\end{align*}

	\textbf{Step Four, final result:} 
	{ 
		Without loss of generality, we assume that $\mu_2\ge\mu_1$. By \Cref{robin-reg-2}, it holds}
	
	\begin{align*}
		&\mu_1\| u\|_{2,\Omega_1} + 	\mu_2\| u\|_{2,\Omega_2} \\
		&\qquad\le  C{ \max\{\mu_1^{-1},\mu_2^{-1},1\}}(1+\gamma_0^{-1})(\mu_1\|\Delta u_1\|_{\Omega_1}
		+\mu_2\|\Delta u_2\|_{\Omega_2} \\
		&\qquad\quad
		+\|	c_0\h^{-1}[\![u]\!]+
		\mu_1\nabla u_1\cdot\bm n\|_{\frac 1 2,\Gamma}
		+\|[\![\mu\nabla u\cdot\bm n]\!]\|_{\frac 1 2,\Gamma})\\
		&\qquad\le C{ \max\{\mu_1^{-1},\mu_2^{-1},1\}}\|f\|_0.\label{est-jh3}
	\end{align*}
	
	{ When $\mu_2\le\mu_1$, we can construct a new auxiliary problem: \begin{align}
			\begin{split}
				-\nabla\cdot(\mu\nabla u)&=f\qquad\;\; \text{in }\Omega,\\
				\gamma[\![u]\!]+\mu_2\nabla u_2\cdot\bm n&=g_D\qquad \text{on }\Gamma,\\
				[\![\mu\nabla u\cdot\bm n]\!]&=g_N\qquad \text{on }\Gamma,
			\end{split}
		\end{align}
		in \Cref{robin-reg1} and \Cref{robin-reg-2} to achieve the same results.}
	
	We have finished  our proof.
\end{proof}
{ 
	\begin{remark}\label{biharmonic_remark}
		For the existence of $v$ in the proof of Step Two, we can consider the following biharmonic boundary value problem:
		\begin{equation}\label{biharmonic}
			\begin{split}
				-\Delta^2 v_i&=0, \quad\text{in } \Omega_i,\\
				\mu_i\nabla v_i\cdot\bm n&=0, \quad\text{on }\Gamma, \\
				v_i&=\beta_\zeta,\quad\text{on }\Gamma,\\
				v_i&=0,\quad\text{on }\partial\Omega_i.
			\end{split}
		\end{equation}
		We set $v|_{\Omega_i}=v_i$, so the existence of $v$ is derived from the existence of problem \eqref{biharmonic}.
	\end{remark}
	\begin{remark}
		Unlike in \Cref{biharmonic_remark}, where we used the existence of a solution to the biharmonic problem to verify the existence of the required $v$ in the proof of Step Two, in this Remark, we present a constructive method to obtain the $v$ we need. Assume that $\Gamma$ is not circular. If $\Gamma$ is circular, the proof is similar.
		
		Denoting $U_\sigma:=\{x\in\Omega, {\rm dist}(x,\Gamma)<\sigma\}$. And construct a projection $\Pi_i^n: \Omega_i\cap U_\sigma\to \Gamma$ such that for any $x=x_0+t\bm n$:
		\begin{align*}
			\Pi_i^nx=x_0,
		\end{align*}
		where $x\in\Omega_i\cap U_\sigma$, $x_0\in\Gamma$ and $0\le t<\sigma$. Then we set $v_\sigma(x)=\beta_\zeta(\Pi_i^n(x))$. It's easy to see $v_\sigma|_\Gamma=\beta_\zeta$. We define $d_{i}(x):=\|x-\Pi_i^n x\|$ and introduce a smooth function $\phi$ to extend $v_\sigma$ to $\Omega_i$:
		\begin{equation*}
			\phi(t)=\begin{cases}
				{\rm exp}\left(1-\frac{1}{1-(\frac{t}{\sigma})^2}\right), & \text{$|t|<\sigma$},\\
				0 & \text{$|t|\ge\sigma$.}
			\end{cases}
		\end{equation*}
		Setting $v_i(x)=v_\sigma\cdot\phi(d_i(x))=\beta_\zeta(\Pi_i^n(x))\cdot\phi(d_i(x))$, where $x\in\Omega_i\cup\Gamma$. Owing to $\phi(d_i(x))|_\Gamma=1$, we obtain $v_i|_\Gamma=v_\sigma|_\Gamma=\beta_\zeta$. By simple calculation, we can get
		\begin{align*}
			\nabla v_i\cdot\bm n=\phi(d_i(x))(\nabla v_\sigma\cdot\bm n)+v_\sigma\cdot\phi'(\nabla d_i(x)\cdot\bm n).
		\end{align*}
		Because of $\phi'(0)=0$, we can get
		\begin{align*}
			\nabla v_i\cdot\bm n|_\Gamma=\nabla v_\sigma\cdot\bm n|_\Gamma
			=\nabla(\beta_\zeta(\Pi_i^n(x)))\cdot\bm n|_\Gamma
			&=D\beta|_{\Pi_i^n}D\Pi_i^n\cdot\bm n|_\Gamma,
		\end{align*}
		where $D$ is derivative operator.
		
		Consider curve $\gamma(t)=x_0+t\bm n$, where $x_0\in\Gamma$ and $0\le t<\sigma$, to get
		\begin{align*}
			\frac{d}{dt}\Pi_i^n(\gamma(t))|_{t=0}=D\Pi_i^n(x_0)\cdot\gamma'(0)=D\Pi_i^n(x_0)\cdot\bm n=0,
		\end{align*}
		which implies $D\Pi_i^n\cdot\bm n|_\Gamma=0$. Then we conclude that $\nabla v_i\cdot\bm n|_\Gamma=0$. Finally, setting 
		\begin{equation}
			v=\begin{cases}
				v_1, &\text{in }\Omega_1,\\
				v_2, &\text{in }\Omega_2.
			\end{cases}
		\end{equation}
		The $v$ constructed in this way is exactly the one we need in the proof of the \Cref{dual-reg-1}.
	\end{remark}
}

\section{Error estimations}\label{section5}  

\subsection{The Scott-Zhang interpolation}
There exists the following Scott-Zhang interpolation (\cite{MR1011446}) $\mathcal I_h$ from $H^1(\Omega)\to V_h$ and it has the following \emph{local} approximation properties. 

\begin{lemma}[\cite{MR1011446}, (4.1)]\label{SZ} For $u$ is smooth enough and any $T\in \mathcal T_h$, it holds
	\begin{align*}
		\|u- \mathcal I_h u\|_{0, T}+h_T	\|\nabla(u- \mathcal I_h u)\|_{0, T}\le Ch_T^{s+1}|u|_{s+1, \omega(T)}
	\end{align*}
	for $s\in [0,k]$, where $\omega(T)$ is the finite union of domains around the element $T$.
\end{lemma}

\begin{lemma}\label{extend}
	For any $s\ge 1$, $i=1,2$,  $v_i\in H^s(\Omega_i)$, there exsits an extention $\widetilde v_i\in H^s(\Omega)\cap H_0^1(\Omega)$ of $v$ such that
	\begin{align*}
		\widetilde v_i|_{\Omega_i}=v_i,\qquad
		\|\widetilde v_i\|_{s}\le \|v_i\|_{s,\Omega_i}.
	\end{align*}
\end{lemma}

For $v\in V$, we define $\mathcal J_hv\in V_h$ as:
\begin{align*}
	\mathcal J_h v|_{\Omega_i}=\mathcal J_h v_i = (\mathcal I_h \widetilde v_i)|_{\Omega_i}, \qquad i=1,2.
\end{align*}

With the above estimation we are ready to get the following one.
\begin{lemma} \label{est-jh} For $i=1,2$, $v\in V$ being smooth enough, it holds
	\begin{align}
		\|v_i-\mathcal J_hv_i\|_{\Omega_i}+{ h\|\nabla(v_i-\mathcal J_hv_i)\|_{\Omega_i}}&\le Ch^{s+1}\|v_i\|_{s+1,\Omega_i},	\\	
		\|\h^{-\frac 1 2}(v_i-\mathcal J_hv_i)\|_{\Gamma}
		&\le 
		Ch^{s}\|v_i\|_{s+1,\Omega_i}	\label{est-jh-partial},\\
		\|\h^{\frac 1 2}\nabla(v_i-\mathcal J_hv_i)\|_{\Gamma}
		&\le 
		Ch^{\ell}\|v_i\|_{\ell+1,\Omega_i}	 \label{est-jh-partial-2}	
	\end{align}
	for all $s\in [0, k]$ and $\ell\in [1, k]$.
\end{lemma}

\begin{proof} 
	We have
	\begin{align*}
		&	\|v_i-\mathcal J_hv_i\|_{\Omega_i}+h\|\nabla(v_i-\mathcal J_hv_i)\|_{\Omega_i}\\
		&\qquad=
		\|\widetilde v_i-\mathcal I_h\widetilde v_i\|_{\Omega_i}+h\|\nabla(\widetilde v_i-\mathcal I_h\widetilde v_i)\|_{\Omega_i}  \\
		&\qquad\le 
		\|\widetilde v_i-\mathcal I_h\widetilde v_i\|_{0}+h\|\nabla(\widetilde v_i-\mathcal I_h\widetilde v_i)\|_{0}\\
		&\qquad\le Ch^{s+1}|\widetilde v_i|_{s+1}\\
		&\qquad \le Ch^{s+1}\|v_i\|_{s+1,\Omega_i}
	\end{align*}
	We use \Cref{SZ}  and  a direct calculation to get
	\begin{align*}
		\|\h^{-\frac 1 2}(v_i-\mathcal J_hv_i)\|_{\Gamma}^2
		&\le C\sum_{E\subset\Gamma\cap T, T\in \mathcal T_h^{\Gamma}}h_T^{-1}\|v_i-\mathcal I_hv_i\|_E^2 \\
		&\le C\sum_{
			\omega=T\cap \Omega_i, 
			T\in \mathcal T_h^{\Gamma}}h_T^{-1}\|v_i-\mathcal I_h {v_i}\|_{\partial \omega}^2 \\
		&\le C\sum_{
			\omega=T\cap \Omega_i, 
			T\in \mathcal T_h^{\Gamma}}(	h_T^{-2}\|v_i- \mathcal J_h v_i\|_{ \omega}^2+	\|\nabla(v_i- \mathcal J_h v_i)\|_{\omega}^2)\\
		&\le C\sum_{
			T\in\mathcal T_h}(	h_T^{-2}\|\widetilde v_i- \mathcal J_h \widetilde v_i\|_{ T}^2+	\|\nabla(\widetilde v_i- \mathcal J_h \widetilde v_i)\|_{T}^2)\\
		&\le  Ch^{2s}|\widetilde v_i|_{s+1}^2\\
		&\le Ch^{2s}\|v_i\|^2_{s+1,\Omega_i}.
	\end{align*}
	Similarity, it holds
	\begin{align*}
		\|\h^{\frac 1 2}\nabla(v_i-\mathcal J_hv_i)\|_{\Gamma}^2
		&=\sum_{E\subset\Gamma\cap T, T\in \mathcal T_h^{\Gamma}}h_E\|\nabla(u-\mathcal J_hu)\|_E^2 \\
		&\le C\sum_{	\omega=T\cap \Omega_i, 
			T\in \mathcal T_h^{\Gamma}}h_T\|\nabla(u-\mathcal J_hu)\|_{\partial  \omega}^2 \\
		&\le C\sum_{	\omega=T\cap \Omega_i, 
			T\in \mathcal T_h^{\Gamma}}h_T\|\nabla(u-\mathcal J_hu)\|_{\partial \omega}^2 \\
		&\le C\sum_{T\in\mathcal T_h}(	\|\nabla(v_i- \mathcal J_h v_i)\|_{0, T}^2+	 {h_T^2\|\nabla^2(v_i- \mathcal J_h v_i)\|_{0, T}^2)}\\
		&\le  Ch^{2\ell}|v_i|_{\ell+1}^2,
	\end{align*}
	where we have used the estimation by the triangle inequality, the $L^2$ bound for $\Pi_k^{o}$, the inverse inequality and the estimation in \Cref{SZ}:
	\begin{align*}
		\|\nabla^2(v_i- \mathcal J_h v_i)\|_{0, T}
		&\le 	|v_i- \Pi_k^ov_i|_{2, T} +|\Pi_k^ov_i- \mathcal J_h v_i|_{2, T}\\
		&\le 	|v_i- \Pi_k^ov_i|_{2, T} +Ch_T^{-2}\|\Pi_k^ov_i- \mathcal J_h v_i\|_{0, T}\\
		&= 	|v_i- \Pi_k^ov_i|_{2, T} +Ch_T^{-2}\|\Pi_k^o(v_i- \mathcal J_h v_i)\|_{0, T}\\
		&\le	Ch_T^{\ell -1}|v_i|_{\ell +1, T} +Ch_T^{-2}\|v_i- \mathcal J_h v_i\|_{0, T}\\
		&\le 	Ch_T^{\ell -1}|v_i|_{\ell +1, \omega(T)}.
	\end{align*}
	
\end{proof}

\begin{lemma} Let $u_h\in V_h$, $u\in V$, $u\in H^{k+1}(\Omega_1)$ and $u\in H^{k+1}(\Omega_2)$, it holds
	\begin{align}
		|a( \mathcal J_h u-u,  \mathcal J_h u-u_h)|
		\le  Ch^k(\mu_1^{\frac 1 2}|u_1|_{k+1,\Omega_1}+\mu_2^{\frac 1 2}|u_2|_{k+1,\Omega_2})\interleave \mathcal J_h u-u_h\interleave.
	\end{align}
\end{lemma}
\begin{proof} We use the definition of $a(\cdot,\cdot)$ to get
	\begin{align*}
		&a( \mathcal J_h u-u,  \mathcal J_h u-u_h) \\
		&\quad=\mu_1(\nabla( \mathcal J_h u-u), \nabla (\mathcal J_h u-u_h))_{\Omega_1}
		+ \mu_2(\nabla ( \mathcal J_h u-u), \nabla (\mathcal J_h u-u_h))_{\Omega_2}\\
		&\quad\quad
		-\langle \{\!\!\{\mu\nabla (\mathcal J_h u-u)\cdot\bm n\}\!\!\},[\![\mathcal J_h u-u_h]\!] \rangle_{\Gamma}
		+\langle [\![\mathcal J_h u-u]\!],\{\!\!\{\mu\nabla (\mathcal J_h u-u_h)\cdot\bm n\}\!\!\} \rangle_{\Gamma}\\
		&\quad\quad
		+\langle c_0\h^{-1}[\![\mathcal J_h u-u]\!],[\![\mathcal J_h u-u_h]\!] \rangle_{\Gamma},\\
		&\quad=:\sum_{i=1}^5 E_i.
	\end{align*}
	{ 
		Then using Cauchy-Schwarz inequality and \Cref{est-jh} estimate the first term $E_1+E_2$
		\begin{align}
			\begin{split}
				|E_1+E_2|&\le \mu_1\|\nabla (\mathcal J_h u_1-u_1)\|_{\Omega_1}
				\|\nabla(\mathcal J_h u_1-u_{h, 1})\|_{\Omega_1} \\
				&\quad + \mu_2\|\nabla (\mathcal J_h u_2-u_2)\|_{\Omega_2}
				\|\nabla(\mathcal J_h u_2-u_{h, 2})\|_{\Omega_2} \\
				&\le Ch^k(\mu_1^{\frac 1 2}|u_1|_{ {k+1,\Omega_1}}
				+\mu_2^{\frac 1 2}|u_2|_{ {k+1,\Omega_2}}
				)
				(\mu_1^{\frac 1 2}\|\nabla(\mathcal J_h u_1-u_{h,1})\|_{\Omega_1} \\
				&\quad+\mu_2^{\frac 1 2}\|\nabla(\mathcal J_h u_2-u_{h,2})\|_{\Omega_2}
				).
			\end{split}
	\end{align}}
	
	By notice that $c_0\ge \mu_1w_1$ and $c_0\ge \mu_2w_2$ we can get
	\begin{equation}\label{lemma5.4{}}
		\begin{split}
			&\| c_0^{-\frac 1 2}\h^{\frac 1 2}\{\!\!\{\mu\nabla (\mathcal J_h u-u)\cdot\bm n\}\!\!\}\|_{\Gamma}\\
			&\le \| c_0^{-\frac 1 2}\h^{\frac 1 2}\mu_1w_1\nabla (\mathcal J_h u_1-u_1)\cdot\bm n\|_{\Gamma}
			+\| c_0^{-\frac 1 2}\h^{\frac 1 2}\mu_2w_2\nabla (\mathcal J_h u_2-u_2)\cdot\bm n\|_{\Gamma} \\
			&\le \|\h^{\frac 12}\mu_1^{\frac 1 2}w_1^{\frac 1 2}\nabla (\mathcal J_h u_1-u_1)\cdot\bm n\|_{\Gamma}
			+ \|\h^{\frac 12}\mu_2^{\frac 1 2}w_2^{\frac 1 2}\nabla (\mathcal J_h u_2-u_2)\cdot\bm n\|_{\Gamma} \\
			&\le  Ch^k(\mu_1^{\frac 1 2}|u_1|_{ {k+1,\Omega_1}}
			+\mu_2^{\frac 1 2}|u_2|_{ {k+1,\Omega_2}}
			).
		\end{split}
	\end{equation}

	So {  adding the term $E_3$, applying the Cauchy-Schwarz inequality, \eqref{lemma5.4{}} and \Cref{est-jh}}
	\begin{align}
		\begin{split}
			|E_3|&\le \| c_0^{-\frac 1 2}\h^{\frac 1 2}\{\!\!\{\mu\nabla (\mathcal J_h u-u)\cdot\bm n\}\!\!\}\|_{\Gamma}\|c_0^{\frac 1 2}\h^{-\frac 1 2}[\![\mathcal J_h u-u_h]\!] \|_{\Gamma} \\
			&\le Ch^{k}(\mu_1^{\frac 1 2}|u_1|_{ {k+1,\Omega_1}}
			+\mu_2^{\frac 1 2}|u_2|_{ {k+1,\Omega_2}}
			)\|c_0^{\frac 1 2}\h^{-\frac 1 2}[\![\mathcal J_h u-u_h]\!] \|_{\Gamma}.
		\end{split}
	\end{align}
	{  For the fourth term $E_4$, we similarly get
		\begin{align}
			\begin{split}
				|E_4|&\le \| c_0^{\frac 1 2}\h^{-\frac 1 2}[\![\mathcal J_h u-u]\!]\|_{\Gamma}\|c_0^{-\frac 1 2}\h^{\frac 1 2}\{\!\!\{\mu\nabla (\mathcal J_h u-u_h)\cdot\bm n\}\!\!\}\|_{\Gamma} \\
				&\le C  \| c_0^{\frac 1 2}\h^{-\frac 1 2}[\![\mathcal J_h u-u]\!]\|_{\Gamma}(\mu_1^{\frac 1 2}\|\nabla(\mathcal J_h u_1-u_{h,1})\|_{\Omega_1}
				+\mu_2^{\frac 1 2}\|\nabla(\mathcal J_h u_2-u_{h,2})\|_{\Omega_2}
				)\\
				&\le Ch^k{ (\mu_1^{\frac{1}{2}}|u_1|_{k+1,\Omega_1}
					+\mu_2^{\frac{1}{2}}|u_2|_{k+1,\Omega_2}
					)}
				(\mu_1^{\frac 1 2}\|\nabla(\mathcal J_h u_1-u_{h,1})\|_{\Omega_1} \\
				&\quad+\mu_2^{\frac 1 2}\|\nabla(\mathcal J_h u_2-u_{h,2})\|_{\Omega_2}
				).
			\end{split}
	\end{align}}
	{  With the help of Cauchy-Schwarz inequality, \eqref{est-jh-partial}, the final term is estimated as}
	\begin{align}
		\begin{split}
			|E_5|&\le
			C\|c_0^{\frac 1 2}\h^{-\frac 1 2}[\![\mathcal J_h u-u]\!]\|_{\Gamma}\|c_0^{\frac 1 2}\h^{-\frac 1 2}[\![\mathcal J_h u-u_h]\!]\|_{\Gamma}  \\
			&\le   Ch^k(\mu_1^{\frac 1 2}|u_1|_{ {k+1,\Omega_1}}
			+\mu_2^{\frac 1 2}|u_2|_{ {k+1,\Omega_2}}
			)\|c_0^{\frac 1 2}\h^{-\frac 1 2}[\![\mathcal J_h u-u_h]\!]\|_{\Gamma}.
		\end{split}
	\end{align}
	{  Combining all the above estimates finishes proof.}
\end{proof}

\subsection{Energy  Error Estimation}

\begin{theorem}[Energy error estimation I] Let $u\in V$ and $u_h\in V_h$ be the solution of \eqref{org} and \eqref{fem}, respectively. Then for $u_1\in H^{k+1}(\Omega)$ and $u_2\in H^{k+1}(\Omega)$, we have
	\begin{align*}
		\interleave u-u_h\interleave
		\le Ch^k(\mu_1^{\frac 1 2}|u_1|_{k+1,\Omega_1}+\mu_2^{\frac 1 2}|u_2|_{k+1,\Omega_2}).
	\end{align*}
\end{theorem}

\begin{proof} We have
	\begin{align*}
		\interleave \mathcal J_h u-u_h\interleave^2
		&= a( \mathcal J_h u-u_h,  \mathcal J_h u-u_h) \\
		&=  a( \mathcal J_h u-u,  \mathcal J_h u-u_h) \\
		&\le Ch^k(\mu_1^{\frac 1 2}\|u_1\|_{k+1,\Omega_1}+\mu_2^{\frac 1 2}\|u_2\|_{k+1,\Omega_2})\interleave \mathcal J_h u-u_h\interleave,
	\end{align*}
	which leads to
	\begin{align}\label{es-en}
		\interleave \mathcal J_h u-u_h\interleave
		\le Ch^k(\mu_1^{\frac 1 2}\|u_1\|_{k+1,\Omega_1}+\mu_2^{\frac 1 2}\|u_2\|_{k+1,\Omega_2}).
	\end{align}
	We use the triangle inequality to get the final result.
\end{proof}

The followoing result is proved in \cite{MR3668542}, with the technique proposed in \cite{MR1941489}.
\begin{theorem}[Energy error estimation II, {\cite[Theorem 4.2 and Remark 5.1]{MR3668542}}] Let $u\in V$ and $u_h\in V_h$ be the solution of \eqref{org} and \eqref{fem}, respectively. Then for $u_1\in H^{k+1}(\Omega)$ and $u_2\in H^{k+1}(\Omega)$, we have
	\begin{align*}
		\mu_1\|\nabla (u_1-u_{h, 1})\|_{\Omega_1}
		+\mu_2\|\nabla(u_2-u_{h,2})\|_{\Omega_2}
		\le Ch^k(\mu_1\|u_1\|_{k+1,\Omega_1}+\mu_2\|u_2\|_{k+1,\Omega_2}).
	\end{align*}
\end{theorem}

	\subsection{$L^2$ error estiamtion}
	
	Before we give the theorem for the $L^2$ error estimation, we first give a estimation for $a(\cdot,\cdot)$.
	\begin{lemma}\label{lem:dual} Let $u\in V$ and $u_h\in V_h$ be the solution of \eqref{org} and \eqref{fem}, respectively. Then for $u_1\in H^{k+1}(\Omega_1)$, $u_2\in H^{k+1}(\Omega_2)$, $\Phi\in V$, $\Phi_1\in H^2(\Omega_1)$
		and $\Phi_2\in H^2(\Omega_2)$, we have
		{ 
			\begin{align*}
				&|a(u-u_h,\Phi-\mathcal J_h\Phi)|
				\\&\le C          h^{k+1}
				(\mu_1^{\frac 1 2}|u_1|_{k+1,\Omega_1}+\mu_2^{\frac 1 2}|u_2|_{k+1,\Omega_2}) (\mu_1^{\frac 1 2}|\Phi_1|_{2,\Omega_1}+
				\mu_2^{\frac 1 2}|\Phi_2|_{2,\Omega_2}
				) .
		\end{align*}}
	\end{lemma}
	\begin{proof} We use the definition of $a(\cdot,\cdot)$ to get
		\begin{align*}
			&a(u-u_h,\Phi-\mathcal J_h\Phi)\\
			&\quad=\mu_1(\nabla (u_1-u_{h,1}), \nabla (\Phi_1-\mathcal J_h\Phi_1) )_{\Omega_1}
			+ \mu_2(\nabla (u_2-u_{h,2}), \nabla (\Phi_2-\mathcal J_h\Phi_2))_{\Omega_2}\\
			&\qquad
			-\langle \{\!\!\{\mu\nabla (u-u_h)\cdot\bm n\}\!\!\},[\![\Phi-\mathcal J_h \Phi]\!] \rangle_{\Gamma}
			+\langle [\![u-u_h]\!],\{\!\!\{\mu\nabla (\Phi-\mathcal J_h\Phi)\cdot\bm n\}\!\!\} \rangle_{\Gamma}\\
			&\qquad
			+\langle c_0\h^{-1}[\![u-u_h]\!],[\![\Phi-\mathcal J_h\Phi]\!] \rangle_{\Gamma}\\
			&\quad=:\sum_{i=1}^5 E_i
		\end{align*}
		We estimate $E_i$ term by term.
		We use the  {Cauchy}- {Schwarz} inequality and the approximation property in \Cref{est-jh} to get 
		\begin{align*}
			\begin{split}
				|E_1+E_2|&\le \mu_1\|\nabla(u_1-u_{h,1})\|_{\Omega_1}
				\|\nabla(\Phi_1-\mathcal J_h\Phi_1)\|_{\Omega_1}\\
				&\quad
				+\mu_2
				\|\nabla(u_2-u_{h,2})\|_{ {\Omega_2}}
				\|\nabla(\Phi_2-\mathcal J_h\Phi_2)\|_{\Omega_2}\\
				&\le C          h^{k+1}
				(\mu_1^{\frac 1 2}|u_1|_{k+1,\Omega_1}+\mu_2^{\frac 1 2}|u_2|_{k+1,\Omega_2}) (\mu_1^{\frac 1 2}|\Phi_1|_{2,\Omega_1}+
				\mu_2^{\frac 1 2}|\Phi_2|_{2,\Omega_2}
				) 
			\end{split}
		\end{align*}
		We use the fact $w_1\mu_1+w_2\mu_2=c_0$ to get
		\begin{align*}
			&\|c_0^{-\frac 1 2}\h^{\frac 1 2} \{\!\!\{\mu\nabla (u-u_h)\cdot\bm n\}\!\!\}\|_{\Gamma}\\
			&\qquad\le 
			\|c_0^{-\frac 1 2}\h^{\frac 1 2} w_1\mu_1\nabla (u_1-u_{h, 1})\cdot\bm n\|_{\Gamma}
			+\|c_0^{-\frac 1 2}\h^{\frac 1 2} w_2\mu_2\nabla (u_2-u_{h,2})\cdot\bm n\|_{\Gamma}\\
			&\qquad\le 
			\|\h^{\frac 1 2}\mu_1^{\frac 1 2}\nabla (u_1-u_{h, 1})\|_{\Gamma}
			+\|\h^{\frac 1 2} \mu_2^{\frac 12}\nabla (u_2-u_{h,2})\|_{\Gamma}\\
		\end{align*}
		We use the triangle inequality, 
		the inverse inequality,
		the estimation \eqref{es-en} and the approximation property for $\mathcal J_h$ in \Cref{est-jh} to get
		\begin{align*}
			&\|c_0^{-\frac 1 2}\h^{\frac 1 2} \{\!\!\{\mu\nabla (u-u_h)\cdot\bm n\}\!\!\}\|_{\Gamma}\\
			&\qquad\le 
			\|\h^{\frac 1 2} \mu_1^{\frac 1 2}\nabla (u_1-\mathcal J_h u_1)\|_{\Gamma}
			+\|\h^{\frac 1 2} \mu_2^{\frac 12}\nabla (u_2-\mathcal J_h u_2)\|_{\Gamma}\\
			&\quad \qquad +
			\|\h^{\frac 1 2} \mu_1^{\frac 1 2}\nabla (\mathcal J_h u_1-u_{h, 1})\|_{\Gamma}
			+\|\h^{\frac 1 2} \mu_2^{\frac 12}\nabla (\mathcal I_hu_2-u_{h,2})\|_{\Gamma}\\
			&\qquad\le 
			\|\h^{\frac 1 2} \mu_1^{\frac 1 2}\nabla (u_1-\mathcal J_h u_1)\|_{\Gamma}
			+\|\h^{\frac 1 2} \mu_2^{\frac 12}\nabla (u_2-\mathcal J_h u_2)\|_{\Gamma}\\
			&\quad \qquad +
			C\|\mu_1^{\frac 1 2}\nabla (\mathcal J_h u_1-u_{h, 1})\|_{\Omega_1}
			+C\|\mu_2^{\frac 12}\nabla (\mathcal I_hu_2-u_{h,2})\|_{\Omega_2}\\
			&\qquad\le Ch^k(\mu_1^{\frac 1 2}|u_1|_{k+1,\Omega_1}+\mu_2^{\frac 1 2}|u_2|_{k+1,\Omega_2}).
		\end{align*}
		With the above estimation, we use the  {Cauchy}- {Schwarz} inequality , the approximation property for $\mathcal J_h$ in \Cref{est-jh}  and the fact
		$c_0\le 2\min(\mu_1,\mu_2)$
		to get 	
		\begin{align*}
			\begin{split}
				|E_3|&\le \|c_0^{-\frac 1 2}\h^{\frac 1 2} \{\!\!\{\mu\nabla (u-u_h)\cdot\bm n\}\!\!\}\|_{\Gamma}\|c_0^{\frac 1 2}\h^{-\frac 1 2}[\![\Phi-\mathcal J_h \Phi]\!]\|_{\Gamma} \\
				&\le C          h^{k+1}
				(\mu_1^{\frac 1 2}|u_1|_{k+1,\Omega_1}+\mu_2^{\frac 1 2}|u_2|_{k+1,\Omega_2}) (\mu_1^{\frac 1 2}|\Phi_1|_{2,\Omega_1}+
				\mu_2^{\frac 1 2}|\Phi_2|_{2,\Omega_2}
				) .
			\end{split}
		\end{align*}
		Similarity, we can get
		\begin{align}
			\begin{split}
				|E_4|&\le \|c_0^{\frac 1 2}\h^{-\frac 1 2}[\![u-u_h]\!]\|_{\Gamma}
				\|c_0^{-\frac 1 2}\{\!\!\{\mu\nabla (\Phi-\mathcal J_h\Phi)\cdot\bm n\}\!\!\}\|_{\Gamma}\\
				&\le C          h^{k+1}
				(\mu_1^{\frac 1 2}|u_1|_{k+1,\Omega_1}+\mu_2^{\frac 1 2}|u_2|_{k+1,\Omega_2}) (\mu_1^{\frac 1 2}|\Phi_1|_{2,\Omega_1}+
				\mu_2^{\frac 1 2}|\Phi_2|_{2,\Omega_2}
				) .
			\end{split}
		\end{align}
		and
		\begin{align}
			\begin{split}
				|E_5|&\le \|c_0^{-\frac 1 2}\h^{-\frac 1 2}[\![u-u_h]\!]\|_{\Gamma}
				\|c_0^{-\frac 1 2}\h^{-\frac 1 2}[\![\Phi-\mathcal J_h\Phi]\!]\|_{\Gamma}\\
				&\le C          h^{k+1}
				(\mu_1^{\frac 1 2}|u_1|_{k+1,\Omega_1}+\mu_2^{\frac 1 2}|u_2|_{k+1,\Omega_2}) (\mu_1^{\frac 1 2}|\Phi_1|_{2,\Omega_1}+
				\mu_2^{\frac 1 2}|\Phi_2|_{2,\Omega_2}
				) .
			\end{split}
		\end{align}
		We use all the estimations above to get finish our proof.
	\end{proof}

	\begin{theorem}[$L^2$ error estimation]\label{a>=1}  Let $u\in V$ and $u_h\in V_h$ be the solution of \eqref{org} and \eqref{fem}, respectively. Then for $u_1\in H^{k+1}(\Omega)$ and $u_2\in H^{k+1}(\Omega)$, we have
		\begin{align*}
			\| u-u_h\|_0
			\le Ch^{k+1}{  \max(\mu_1^{-\frac 1 2},\mu_1^{-\frac{3}{2}},\mu_2^{-\frac 1 2},\mu_2^{-\frac{3}{2}})}
			(\mu_1^{\frac 1 2}|u_1|_{k+1,\Omega_1}+\mu_2^{\frac 1 2}|u_2|_{k+1,\Omega_2}).
		\end{align*}
		{ In particular, if $\mu_1,\mu_2\ge1$, we have
			\begin{align*}
				\| u-u_h\|_0
				\le Ch^{k+1}\max(\mu_1^{-\frac 1 2},\mu_2^{-\frac 1 2})
				(\mu_1^{\frac 1 2}|u_1|_{k+1,\Omega_1}+\mu_2^{\frac 1 2}|u_2|_{k+1,\Omega_2}).
		\end{align*}}
	\end{theorem}
	\begin{proof} We define the dual problem: Find $\Phi\in V$ such that
		\begin{align}\label{dual}
			a^\star(\Phi, v)
			=(u-u_h, v)
		\end{align}
		holds for all $v\in V$. Then by \Cref{dual-reg-1}, the above problem has a unique solution $u\in V$ satisfying
		\begin{align}\label{dualreg}
			\mu_1\|\Phi_1\|_{2,\Omega_1}+\mu_2\|\Phi_2\|_{2,\Omega_2}\le  C{ \max\{\mu_1^{-1},\mu_2^{-1},1\}}\|u-u_h\|_0.
		\end{align}
		We take $v=u-u_h$ in \eqref{dual} ,  the property \eqref{sys} and  the orthogonality in \Cref{orth} to get
		\begin{align*}
			\|u-u_h\|_0^2 =	a^\star(\Phi, u-u_h) 
			=a(u-u_h,\Phi)
			=a(u-u_h,\Phi-\mathcal J_h\Phi).
		\end{align*}
		We use the result in \Cref{lem:dual}  and the regularity estimation \eqref{dualreg} to get
		\begin{align*} 
			&\|u-u_h\|_0^2            \\
			&\le         C h^{k+1}
			(\mu_1^{\frac 1 2}|u_1|_{k+1,\Omega_1}+\mu_2^{\frac 1 2}|u_2|_{k+1,\Omega_2}) (\mu_1^{\frac 1 2}|\Phi_1|_{2,\Omega_1}+
			\mu_2^{\frac 1 2}|\Phi_2|_{2,\Omega_2}
			)\\
			&
			\le         C h^{k+1}
			\max(\mu_1^{-\frac 1 2},\mu_2^{-\frac 1 2})
			(\mu_1^{\frac 1 2}|u_1|_{k+1,\Omega_1}+\mu_2^{\frac 1 2}|u_2|_{k+1,\Omega_2}) (\mu_1\|\Phi_1\|_{2,\Omega_1}+
			\mu_\|\Phi_2\|_{2,\Omega_2}
			)\\
			&\le 
			C          h^{k+1}
			\max(\mu_1^{-\frac 1 2},\mu_2^{-\frac 1 2}){ \max(\mu_1^{-1},\mu_2^{-1},1)}
			(\mu_1^{\frac 1 2}|u_1|_{k+1,\Omega_1}+\mu_2^{\frac 1 2}|u_2|_{k+1,\Omega_2}) \|u-u_h\|_0\\
			&={ C          h^{k+1}
				\max(\mu_1^{-\frac 1 2},\mu_1^{-\frac{3}{2}},\mu_2^{-\frac 1 2},\mu_2^{-\frac{3}{2}})
				(\mu_1^{\frac 1 2}|u_1|_{k+1,\Omega_1}+\mu_2^{\frac 1 2}|u_2|_{k+1,\Omega_2}) \|u-u_h\|_0}.
		\end{align*}
	\end{proof}
	
	\section{Error estimations for penalty-free}	
	In this section, we mainly discuss an optimal $L^2$-norm estimate for \Cref{fem-org} in the penalty-free case.
	
	We define 
	\begin{align}
		\begin{split}
			a_0(u, v)
			&=\mu_1(\nabla u_1, \nabla v_1)_{\Omega_1}
			+ \mu_2(\nabla u_2, \nabla v_2)_{\Omega_2}\\
			&\quad
			-\langle \{\!\!\{\mu\nabla u\cdot\bm n\}\!\!\},[\![v]\!] \rangle_{\Gamma}
			+\langle [\![u]\!],\{\!\!\{\mu\nabla v\cdot\bm n\}\!\!\} \rangle_{\Gamma},
		\end{split}
	\end{align}
	and
	\begin{align}
		\ell_0(v)=(f, v)+\langle  g_N, w_2v_1+w_1v_2\rangle_{\Gamma}
		+\langle
		g_D,\{\!\!\{\mu\nabla v\cdot\bm n\}\!\!\}
		\rangle_{\Gamma}.
	\end{align}
	Then by \Cref{lem:por-averge} we know that, the  solution of \eqref{org} are also the solution of the follow problem:  Find $u\in V$ such that
	\begin{align}\label{fem-org-penaltyfree}
		a_0(u, v)= \ell_0(v)
	\end{align}
	holds for all $v\in V$.

	{ 
		The finite element method reads: Find $u_h\in V_h$ such that
		\begin{align} \label{fem-penalty-free}
			a_0(u_h, v_h)+s_h(u_h,v_h)=\ell_0(v_h),
		\end{align}
		where the operator $s_h$ is the ghost penalty \cite{MR2738930}, defined as
		\begin{align*}
			s_h(u_h,v_h)&=\gamma_g\sum_{i=1}^2\sum_{E\in\mathfrak{E}^i_\Gamma}\sum_{l=1}^{k}\langle\mu_i\h^{2l-1}[\![D_{n_E}^l u_h^i]\!],[\![D_{n_E}^l v_h^i]\!]\rangle_E.
		\end{align*}
		This penalisation allows the condition number of the matrix can be independent of how the domain boundary intersects the computational mesh. The sets $\mathfrak{E}^i_\Gamma$ for $i=1,2$ is defined as 
		\begin{align*}
			\mathfrak{E}^i_\Gamma&=\{ E\subset\partial T |T\in \mathcal{T}_h^\Gamma, E\cap\Omega_i\neq\emptyset\}.
		\end{align*}
		Here $D_{n_E}^l$ is the partial derivative of order $l$ in the direction $n_E$ and we assume $\gamma_g=\mathcal{O}(1)$.
		
		It's easy to get the orthogonality
		\begin{align}\label{a_0or}
			a_0(u-u_h,v_h)-s_h(u_h,v_h)=0,\qquad s_h(u,v_h)=0,
		\end{align}
		for all $v_h\in V_h$. For any $v\in V$, we define the norm
		\begin{align*}
			\interleave v\interleave_*^2&=\mu_1\|\nabla v_1\|_{\Omega_1}^2
			+\mu_2\|\nabla v_2\|_{\Omega_2}^2
			+c_0\|\h^{-\frac 1 2}[\![v]\!]\|_{\Gamma}^2+s_h(v,v),\\
			\|v\|_*^2&=\interleave v\interleave_*^2+c_0^{-1}\|\h^{\frac{1}{2}}\{\!\!\{\mu\nabla v\cdot\bm n\}\!\!\}\|_\Gamma^2.
		\end{align*}
		Using the definition of $a(\cdot, \cdot)$ and the inverse inequality, we can prove that the bilinear form $a_0(\cdot, \cdot)$ is continuous with respect to the above norms. For any $u_h, v_h \in V_h$, it holds that
		\begin{align*}
			|a_0(u_h,v_h)|&\le C(
			\mu_1\|\nabla u_{h, 1}\|_{\Omega_1}\|\nabla v_{h,1}\|_{\Omega_1} +
			\mu_2\|\nabla u_{h, 2}\|_{\Omega_2}\|\nabla v_{h,2}\|_{\Omega_2}
			\\&\quad+c_0^{-\frac{1}{2}}\|\h^{\frac 1 2}\{\!\!\{\mu\nabla u_h\cdot\bm n\}\!\!\}\|_{\Gamma}c_0^{\frac{
					1
				}{2}}\|\h^{-\frac{1}{2}}[\![v_h]\!]\|_\Gamma\\
			&\quad+(\sum_{i=1}^2(w_i\mu_i)^{\frac{1}{2}}\|\nabla v_{h,i}\|_{\Omega_i})c_0^{\frac{1}{2}}\|\h^{-\frac 1 2}[\![u_h]\!]\|_{\Gamma}\\
			&\le M	\| u_h\|_* \cdot	\interleave v_h\interleave_*.
		\end{align*}
	}
	\begin{definition} For $u, v\in V$, we define the adjoint form of $a_0$ as
		\begin{align}
			\begin{split}
				a^\star_0(u, v)
				&=\mu_1(\nabla u_1, \nabla v_1)_{\Omega_1}
				+ \mu_2(\nabla u_2, \nabla v_2)_{\Omega_2}\\
				&\quad
				+\langle \{\!\!\{\mu\nabla u\cdot\bm n\}\!\!\},[\![v]\!] \rangle_{\Gamma}
				-\langle [\![u]\!],\{\!\!\{\mu\nabla v\cdot\bm n\}\!\!\} \rangle_{\Gamma}.
			\end{split}
		\end{align}
		It is straightforward to verify that the following relation holds: 
		\begin{align}\label{sys-pfree}
			a_0(v ,u)=a^{\star}_0(u, v).
		\end{align}
	\end{definition}
	
	\subsection{Stability}
	In this subsection, we will present the stability of the bilinear form $a_0(\cdot,\cdot)$. 
	
	\begin{theorem}
		\label{thm:a0lbb}
		There exists a positive constant $\beta > 0$ such that for any given $u \in V$ the following inf-sup conditions hold
		\begin{align}
			\sup_{0\neq v\in V} \frac{a_0^{\star}(u, v)}{\|v\|_1}\ge C\|u\|_1, \label{con-infsup-1}\\
			\sup_{0\neq v\in V} \frac{a_0(u, v)}{\|v\|_1}\ge C\|u\|_1.\label{con-infsup-2}
		\end{align}
	\end{theorem}
	\begin{proof} 
		\textbf{Proof of \eqref{con-infsup-1}:}	
		For any given $u\in V$, consider $w\in V$ solving: 
		\begin{align*}
			-\nabla\cdot(\mu \nabla w)&= u\qquad \text{in }\Omega,\\
			[\![\mu \nabla w\cdot\bm n]\!]&=0\qquad \text{on }\Gamma,\\
			[\![w]\!]&=0\qquad\!\! \text{on }\Gamma.
		\end{align*}
		By the elliptic regularity we can get
		\begin{align}\label{con-bound-w-1}
			\|w_{i}\|_{1,\Omega_i}\le  C\|u\|_{0,\Omega_i}.
		\end{align}
		Integration by parts, \eqref{lem:por-averge}, we can obtain
		\begin{align}\label{con-a-w-1}
			a^{\star}_0(u, w)=\sum_i^2-\mu_i(u_i,\Delta w_i)_{\Omega_i}=\sum_i^2\|u_i\|_{0,\Omega_i}^2.
		\end{align}
		Defining $v=u+w\in H^1(\Omega)$ and using \eqref{con-bound-w-1}, we have
		\begin{align}\label{con-boud-v-1}
			\|v\|_{1,\Omega_i}\le \|u\|_{1,\Omega_i}+\|w\|_{1,\Omega_i}\le C\|u\|_{1,\Omega_i}.
		\end{align}
		Then using \eqref{con-a-w-1} and \eqref{con-boud-v-1} implies
		\begin{align*}
			a^{\star}_0(u,v)&= a^{\star}_0(u, u)+a_0^{\star}(u, w)=\sum_{i=1}^2\left(\mu_i\|\nabla u_i\|_{0,\Omega_i}^2 +\|u_i\|_{0,\Omega_i}^2\right)\\
			&=\sum_{i=1}^2\mu_i\|u\|_{1,\Omega_i}^2\ge C\min\{\mu_1,\mu_2\}\|u\|_1\|v\|_1,
		\end{align*}
		which proves \eqref{con-infsup-1}.
		
		\textbf{Proof of \eqref{con-infsup-2}:} 
		When $u$ is a constant, \eqref{con-infsup-2} obviously holds by Poincaré inequality. Clearly, for any given $u\in V$, where $u$ isn't a constant, it holds
		\begin{align}\label{con-infsup-3}
			\sup_{0\neq v\in V}
			a^{\star}_0(v, u)> 0.
		\end{align}
		In fact, if there is a $0\neq u\in V$ such that
		\begin{align*}
			\sup_{0\neq v\in V}
			a^{\star}_0(v, u) =0.
		\end{align*}
		Then
		\begin{align*}
			\sum_{i=1}^2\mu_i\|\nabla u_i\|_{0,\Omega_i}^2=a^{\star}_0(u, u) =0,
		\end{align*}
		implying $u$ is a piecewise constant. For any $v\in V$,
		\begin{align*}
			a^{\star}_0(v, u)=\langle\{\mu\nabla v\cdot\bm n\},\![\![ u]\!]\!\ \rangle_{\Gamma}
			=\![\![ u]\!]\!\ \langle\{\mu\nabla v\cdot\bm n\}, 1 \rangle_{\Gamma}=0.
		\end{align*}
		By the arbitrary of $v\in V$ we can deduce $u$ is a constant, 
		which cases contradiction. {\eqref{con-infsup-3} holds.}
		{Given} $u\in H^1(\Omega)$, find $w\in H^1(\Omega)$ such that
		\begin{align}\label{con-def-w-2}
			a^{\star}(w, v)=(u, {v}).
		\end{align}
		The inf-sup condition \eqref{con-infsup-1} and the Babuška  Theorem ensure  a unique $w\in V$ for this problem with
		\begin{align}\label{con-bound-w-2}
			C\|w\|_1\le 
			\sup_{0\neq v\in V} \frac{a^{\star}_0(w, v)}{\|v\|_1}
			=\sup_{0\neq v\in V} \frac{({u},v)}{\|v\|_1}
			\le C\|u\|_0.
		\end{align}
		{Again, defining $v=u+w$ and using \eqref{con-bound-w-2} give}
		\begin{align}\label{con-bound-v-2}
			\|v\|_1\le \|u\|_1+\|w\|_1\le C\|u\|_1.
		\end{align}
		{In the end, using \eqref{con-bound-w-2} and \eqref{con-bound-v-2}, we write}
		\begin{align*}
			a_0(u, v)=a_0(u, u)+a_0(u, w)= a_0(u, u)+a_0^{\star}(v, u)\ge C\min\{\mu_1,\mu_2\}\|u\|_1\|v\|_1,
		\end{align*}
		concluding the proof.
	\end{proof}
	
	To facilitate the subsequent analysis of the stability and error estimates for the penalty-free non-symmetric Nitsche's method, we establish some notations and assumptions concerning the geometric structure, primarily following the framework of \cite{boiveau2015fitted}. We define the sets of elements related to the interface as
	
	
	\begin{align*}
		\Omega_i^*=\{T\in\mathcal{T}_h|T\cap\Omega_i\neq\emptyset\}.
	\end{align*}
	Furthermore, we partition the set of elements intersected by the interface, $\mathcal{T}_h^\Gamma$, into $N_p$ disjoint subsets $\mathcal{T}_j^\Gamma$, where $j = 1, \ldots, N_p$.
	Let $I_{\mathcal{T}_j^\Gamma}$ be the index set of all nodes ${x_n}$ in the subset $\mathcal{T}_j^\Gamma$. We define two nodal subsets $I_j^1$ and $I_j^2$ as follows:
	\begin{align*}
		I_j^1 := \{ x_n \in I_{\mathcal{T}_j^\Gamma} \mid x_n \in \Omega_1, \, x_n \notin I_{\mathcal{T}_i^\Gamma} \, \forall i \neq j \},\\
		I_j^2 := \{ x_n \in I_{\mathcal{T}_j^\Gamma} \mid x_n \in \Omega_2, \, x_n \notin I_{\mathcal{T}_i^\Gamma} \, \forall i \neq j \},
	\end{align*}
	Based on these nodal sets, we construct two patches $P_j^1 $ and $P_j^2$ for each $\mathcal{T}_j^\Gamma$:
	\begin{align*}
		P_j^1 := \mathcal{T}_j^\Gamma \cup \{ T \in \mathcal{T}_h \mid I_j^1 \in T \}, \quad P_j^2 := \mathcal{T}_j^\Gamma \cup \{ T \in \mathcal{T}_h \mid I_j^2 \in T \}.
	\end{align*}
	Each patch $P_j^i$ is constructed such that $I_j^i \neq \emptyset$ for $i = 1,2$. Let $\Gamma_j := \Gamma \cap \mathcal{T}_j^\Gamma$ denote the part of the interface contained within the patches $P_j^1$ and $P_j^2$. For all $j$ and $i = 1, 2$, the patch $P_j^i$ satisfies the following scaling properties:
	\begin{align*}
		|\Gamma_j|\sim h^{d-1}  \quad \text{and}\quad  |P_j^i| \sim h^d.
	\end{align*}
	
	The key to proving stability is to construct a suitable test function that controls the interface terms. Following the methodology of \cite{boiveau2015fitted}, we employ the patch structure defined above. We define a function $v_h^1 \in V_h$ as
	\begin{align}\label{v_Gamma^1}
		v_{\Gamma,1} = \alpha \sum_{j=1}^{N_p} v_{j,1},
	\end{align}
	where
	\begin{align*}
		v_{j,1} = \nu_j \phi_j \ , \quad \nu_j \in \mathbb{R}.,
	\end{align*}
	and the function $\phi_j$ associated with the patch $P_j^1$ is defined by its nodal values:
	\begin{equation*}
		\phi_j(x_i) = 
		\begin{cases} 
			0, & \text{for } x_i \notin I_j^1, \\
			1, & \text{for } x_i \in I_j^1,
		\end{cases}
	\end{equation*}
	with $ i = 1, \ldots, N_n $. $ N_n $ is the number of nodes in the triangulation $ \mathcal{T}_h $. 
	
	Let $\overline{w}^{\Gamma_j} := |\Gamma_j|^{-1} \int_{\Gamma_j} w \, ds$ denote the $P_0$-projection of a function $w$ onto the interface segment $\Gamma_j$. The function $v_{j,1}$ is constructed to have the crucial property:
	\begin{align}\label{means-vj}
		|\Gamma_j|^{-1} \int_{\Gamma_j} \nabla v_{j,1} \cdot n \, {\rm d}s = h^{-1} \overline{\![\![ u_h \!]\!]}^{\Gamma_j}
	\end{align}
	
	We recall the following estimates from \cite{boiveau2015fitted}, which are essential for the analysis
	\begin{align}
		\left\| u_{h,i} - \overline{u_{h,i}}^{\Gamma_j} \right\|_{\Gamma_j} &\le Ch \, \| \nabla u_{h,i} \|_{\Gamma_j}\label{means1},\\
		\left\| v_{j,1} \right\|_{P_j^1} &\le Ch \, \| \nabla v_{j,1} \|_{P_j^1}\label{patch-est},\\
		\left\| \nabla v_{j,1} \right\|_{P_j^1}^2 &\le C \left\| \mathfrak{h}^{-\frac{1}{2}} \overline{\llbracket u_h \rrbracket}^{\Gamma_j} \right\|_{\Gamma_j}^2\label{brack-est}. 
	\end{align}
	
	Then following the procedure for proving Theorem 1 in \cite{boiveau2015fitted}, we similarly complete the proof of the stability of the bilinear form in \Cref{lemma:6.1} and \Cref{thm:dis-lbb2}.
	{ \begin{lemma}\label{lemma:6.1}
			Considering the patches $P_j^i$ as defined above, for all $u_h \in V_h$, the following inequality holds
			\begin{equation}
				\begin{split}
					&\sum_{j=1}^{N_p}c_0 \left\| \mathfrak{h}^{-\frac{1}{2}} \overline{[\![ u_h ]\!]}^{\Gamma_j} \right\|^2_{\Gamma_j}\\
					& \qquad\ge \sum_{j=1}^{N_p}\frac{c_0}{2} \left\| \mathfrak{h}^{-\frac{1}{2}} [\![ u_h ]\!] \right\|_{\Gamma_j}^2 - C w_1 \sum_{j=1}^{N_p^1} \left\| \mu_1^{\frac{1}{2}} \nabla u_{h,1} \right\|_{P_j^1}^2 - C w_2 \sum_{j=1}^{N_p^2} \left\| \mu_2^{\frac{1}{2}} \nabla u_{h,2} \right\|_{P_j^2}^2.
				\end{split}
			\end{equation}
		\end{lemma}
		
		\begin{proof}
			With the help of triangle inequality, the definition of the jump, inequality \eqref{means1}, trace inequality and inverse inequality we can write
			\begin{align*}
				& c_0^{\frac{1}{2}}\left\| \mathfrak{h}^{-\frac{1}{2}} [\![ u_h ]\!] \right\|_{\Gamma_j}\\
				&\le c_0^{\frac{1}{2}}\left\| \mathfrak{h}^{-\frac{1}{2}} \overline{[\![ u_h ]\!]}^{\Gamma_j} \right\|_{\Gamma_j} + c_0^{\frac{1}{2}}\left\| \mathfrak{h}^{-\frac{1}{2}} \left( [\![ u_h ]\!] - \overline{[\![ u_h ]\!]}^{\Gamma_j} \right) \right\|_{\Gamma_j^1} \\
				&= c_0^{\frac{1}{2}} \left\| \mathfrak{h}^{-\frac{1}{2}} \overline{[\![ u_h ]\!]}^{\Gamma_j} \right\|_{\Gamma_j} +   c_0^{\frac{1}{2}}\left\|\mathfrak{h}^{-\frac{1}{2}}\left( u_{h,1} - \overline{u_{h,1}}^{\Gamma_j^1}\right) -\left(u_{h,2} - \overline{u_{h,2}}^{\Gamma_j^2}\right) \right\|_{\Gamma_j} \\
				&\leq c_0^{\frac{1}{2}}\left\| \mathfrak{h}^{-\frac{1}{2}} \overline{[\![ u_h ]\!]}^{\Gamma_j} \right\|_{\Gamma_j} +   c_0^{\frac{1}{2}} \left\|\mathfrak{h}^{-\frac{1}{2}}\left( u_{h,1} - \overline{u_{h,1}}^{\Gamma_j^1}\right) \right\|_{\Gamma_j^1}+ c_0^{\frac{1}{2}}\left\| \mathfrak{h}^{-\frac{1}{2}} \left(u_{h,2} - \overline{u_{h,2}}^{\Gamma_j^2}\right) \right\|_{\Gamma_j^2} \\
				&\leq  c_0^{\frac{1}{2}}\left\| \mathfrak{h}^{-\frac{1}{2}} \overline{[\![ u_h ]\!]}^{\Gamma_j} \right\|_{\Gamma_j} + C w_1^{\frac{1}{2}} \left\| \mu_1^{\frac{1}{2}} \nabla u_{h,1} \right\|_{P_j^1} + C w_2^{\frac{1}{2}}  \left\| \mu_2^{\frac{1}{2}} \nabla u_{h,2} \right\|_{P_j^2}.
			\end{align*}
			Square both the left and right sides, then apply Young's inequality to the right side to obtain
			\begin{align*}
				c_0 \left\| \mathfrak{h}^{-\frac{1}{2}} [\![ u_h ]\!] \right\|_{\Gamma_j}^2 &\le 2c_0 \left\| \mathfrak{h}^{-\frac{1}{2}} \overline{[\![ u_h ]\!]}^{\Gamma_j} \right\|^2_{\Gamma_j} + C w_1  \left\| \mu_1^{\frac{1}{2}} \nabla u_{h,1} \right\|_{P_j^1}^2 + C w_2  \left\| \mu_2^{\frac{1}{2}} \nabla u_{h,2} \right\|_{P_j^2}^2,
			\end{align*}
			which finishes proof.
		\end{proof}
		
		\begin{theorem}
			\label{thm:dis-lbb2}
			There exists a positive constant $\beta > 0$ such that for all $u_h \in V_h$ the following inequality holds
			\begin{align}
				\beta \interleave u_h \interleave_* \leqslant \sup_{v_h \in V_h} \frac{a^{\star}_0(u_h, v_h)+s_h(u_h,v_h)}{\interleave v_h \interleave_*}\label{dis-lbb-adjonit}\\
				\beta \interleave u_h \interleave_* \leqslant \sup_{v_h \in V_h} \frac{a_0(u_h, v_h)+s_h(u_h,v_h)}{\interleave v_h \interleave_*}\label{dis-lbb}.
			\end{align}
		\end{theorem}
		\begin{proof}
			Without loss of generality, we assume that $\mu_1\le\mu_2$. When $\mu_1\ge\mu_2$, we can construct a new interpolation $v_{\Gamma,2}=\alpha\sum_{j=1}^{N_p}v_{j,2}$ to achieve the same result.
			We only give the proof for \eqref{dis-lbb}, since the proof for \eqref{dis-lbb-adjonit} is similar.
			
			\textbf{Step One:}
			First of all, for any $u_h, v_h \in V_h$ with $v_h = u_h + v_{\Gamma,1}$, $v_{\Gamma,1}$ defined by \Cref{v_Gamma^1}, we proof that there exists a positive constant $\beta_0$ such that the following inequality holds
			\begin{align*}
				\beta_0 \interleave u_h \interleave^2_* \leqslant a_0(u_h, v_h)+s_h(u_h,v_h).
			\end{align*}
			By the definition of  $v_{\Gamma,1}$, we can write the following
			\begin{align*}
				(a_0+s_h)(u_h, v_h) = (a_0+s_h)(u_h, u_h) + \alpha \sum_{j=1}^{N_p^1} (a_0(u_h, v_{j,1})+s_h(u_h,v_{j,1})).
			\end{align*}
			Clearly it holds
			\begin{align*}
				(a_0+s_h)(u_h, u_h) &= \left\| \mu_1^{\frac{1}{2}} \nabla u_{h,1} \right\|_{\Omega_1}^2 + \left\| \mu_2^{\frac{1}{2}} \nabla u_{h,2} \right\|_{\Omega_2}^2+s_h(u_,u_h)\\
				&\ge C(\left\| \mu_1^{\frac{1}{2}} \nabla u_{h,1} \right\|_{\Omega_1^*}^2 + \left\| \mu_2^{\frac{1}{2}} \nabla u_{h,2} \right\|_{\Omega_2^*}^2)
			\end{align*}
			and
			\begin{align*}
				(a_0+s_h)(u_h, v_{j,1}) &= (\mu_1 \nabla u_{h,1}, \nabla v_{j,1})_{P_j^1\cap\Omega_1}- \langle \{ \mu \nabla u_h \cdot n \}, v_{j,1} \rangle_{\Gamma_j^1}  \\
				&\quad+\omega_1 \langle \mu_1 \nabla v_{j,1} \cdot n, [\![ u_h ]\!] \rangle_{\Gamma_j^1}+s_h(u_{h,1},v_{j,1}).
			\end{align*}
			\textbf{Step Two, estimate the lower bound of $(\mu_1 \nabla u_{h,1}, \alpha \nabla v_{j,1})_{P_j^1}+s_h(u_{h,1},v_j^1)$:} 
			Using Cauchy-Schwarz inequality and inequality \eqref{brack-est}, we can get
			
			\begin{align*}
				&(\mu_1 \nabla u_{h,1}, \alpha \nabla v_{j,1})_{P_j^1}+s_h(u_{h,1},v_{j,1})\\ &\qquad\ge  -\left\| \mu_1^{\frac{1}{2}} \nabla u_{h,1} \right\|_{P_j^1} \alpha \mu_1^{\frac{1}{2}} \| \nabla v_{j,1} \|_{P_j^1}-s_h(u_{h,1},u_{h,1})^{\frac{1}{2}}\alpha s_h(v_{j,1},v_{j,1})^{\frac{1}{2}} \\
				&\qquad\geq -\epsilon \left\| \mu_1^{\frac{1}{2}} \nabla u_{h,1} \right\|_{P_j^1}^2 - \frac{C \alpha^2}{4\epsilon} \left\| \nabla v_{j,1} \right\|_{P_j^1}^2\\
				&\qquad\geq -\epsilon \left\| \mu_1^{\frac{1}{2}} \nabla u_{h,1} \right\|_{P_j^1}^2 - \frac{C \alpha^2c_0}{4\epsilon}\left(1+\frac{\mu_1}{\mu_2}\right) \left\|\h^{-\frac{1}{2}} \overline{[\![ u_h ]\!]}^{\Gamma^1_j} \right\|_{\Gamma_j^1}^2.
			\end{align*}
			The derivations in the last two lines are derived from \cite[Proposition 5.1]{MR3268662}:
			\begin{align*}
				C_1\|\mu_i^{\frac{1}{2}}\nabla v_{h,i}\|^2_{\Omega_i^*}\le\|\mu_i^{\frac{1}{2}}\nabla v_{h,i}\|^2_{\Omega_i}+s_h(v_{h,i},v_{h,i})\le C_2\|\mu_i^{\frac{1}{2}}\nabla v_{h,i}\|^2_{\Omega_i^*}.
			\end{align*}
			\textbf{Step Three, estimate the upper bound of $\left\langle \{\mu\nabla u_h\cdot n\},\alpha v_{j,1} \right\rangle_{\Gamma_j^1}$:} 
			Using the trace and inverse inequalities, \eqref{means1} and \eqref{brack-est} we can write
			\begin{align*}
				&\left\langle \{\!\!\{\mu\nabla u_h\cdot n\}\!\!\},\alpha v_{j,1} \right\rangle_{\Gamma_j^1}\\ 
				&\le \left( (\omega_1\mu_1)^{\frac{1}{2}}\|\mu_1^{\frac{1}{2}}\nabla u_{h,1}\cdot n\|_{\Gamma_j^1} + (\omega_2\mu_2)^{\frac{1}{2}}\|\mu_2^{\frac{1}{2}}\nabla u_{h,2}\cdot n\|_{\Gamma_j^1} \right)\alpha (c_0h)^{\frac{1}{2}}\|\mathfrak{h}^{-\frac{1}{2}} v_{j,1}\|_{\Gamma_j^1} \\
				&\le  \left( (\omega_1\mu_1)^{\frac{1}{2}}\|\mu_1^{\frac{1}{2}}\nabla u_{h,1}\cdot n\|_{\Gamma_j^1} + (\omega_2\mu_2)^{\frac{1}{2}}\|\mu_2^{\frac{1}{2}}\nabla u_{h,2}\cdot n\|_{\Gamma_j^1} \right)\alpha (c_0h)^{\frac{1}{2}}\left\|\mathfrak{h}^{-\frac{1}{2}}\overline{[\![u_h]\!]}^{\Gamma_j^1}\right\|_{\Gamma_j^1}.
			\end{align*}
			
			Taking the sum over the full boundary $\Gamma$ and using trace and inverse inequalities once again we obtain
			\begin{align*}
				\sum_{j=1}^{N_p} \left\langle \{\mu\nabla u_h\cdot n\},\alpha v_{j,1} \right\rangle_{\Gamma_j^1} &\leq \frac{\alpha^2}{2\epsilon} c_0\sum_{j=1}^{N_p^1} \left\|\mathfrak{h}^{-\frac{1}{2}}[\![u_h]\!]^{\Gamma_j^1}\right\|_{\Gamma_j^1}^2 + \frac{\epsilon\omega_1\mu_1 h_1}{2} \sum_{j=1}^{N_p^1} \|\mu_1^{\frac{1}{2}}\nabla u_{h,1}\cdot n\|_{\Gamma_j^1}^2 \\
				&\quad + \frac{\epsilon\omega_2\mu_2 h_2}{2} \sum_{j=1}^{N_p^2} \|\mu_2^{\frac{1}{2}}\nabla u_{h,2}\cdot n\|_{\Gamma_j^2}^2 \\
				&\le C\frac{\alpha^2}{2\epsilon}c_0 \sum_{j=1}^{N_p^1} \left\|\mathfrak{h}^{-\frac{1}{2}}[\![u_h]\!]^{\Gamma_j^1}\right\|_{\Gamma_j^1}^2 + \epsilon\omega_1\mu_1 \sum_{j=1}^{N_p^1} \left\|\mu_1^{\frac{1}{2}}\nabla u_{h,1}\right\|_{P_j^1}^2 \\
				&\quad + \epsilon\omega_2\mu_2 \sum_{j=1}^{N_p^2} \left\|\mu_2^{\frac{1}{2}}\nabla u_{h,2}\right\|_{P_j^2}^2.
			\end{align*}
			
			\textbf{Step Four, estimate the upper bound of $\alpha\omega_1 \left\langle \mu_1\nabla v_{j,1}\cdot n, [\![u_h]\!] \right\rangle_{\Gamma_j^1}$:} 
			Using the property \eqref{means-vj} of $v_{j,1}$ we can write for each face $\Gamma_j^1$, Using the trace inequality, Young's inequality and \eqref{means1} and inequality \eqref{brack-est} to get
			\begin{align*}
				&\alpha\omega_1 \left\langle \mu_1\nabla v_{j,1}\cdot n, [\![u_h]\!] \right\rangle_{\Gamma_j}\\
				&\qquad=\frac{\alpha c_0}{2} \left\|\mathfrak{h}^{-\frac{1}{2}}[\![u_h]\!]^{\Gamma_j}\right\|_{\Gamma_j}^2 + \alpha\omega_1 \left\langle \mu_1\nabla v_j^1\cdot n, [\![u_h]\!] - \overline{[\![u_h]\!]}^{\Gamma_j} \right\rangle_{\Gamma_j}\\
				&\qquad\ge \frac{\alpha c_0}{2} \left\|\mathfrak{h}^{-\frac{1}{2}}[\![u_h]\!]^{\Gamma_j}\right\|_{\Gamma_j}- \alpha\omega_1\mu_1 \|\nabla v_{j,1}\cdot n\|_{\Gamma_j^1} \left\|[\![u_h]\!] - \overline{[\![u_h]\!]}^{\Gamma_j^1}\right\|_{\Gamma_j^1}\\
				&\qquad\ge\alpha(1-\frac{C\alpha}{2\epsilon})\frac{c_0}{2}\left\|\mathfrak{h}^{-\frac{1}{2}}[\![u_h]\!]^{\Gamma_j}\right\|_{\Gamma_j}-\epsilon w_1\mu_1\|\nabla u_{h,1}\|_{P_j^1}^2-\epsilon w_2\mu_2\|\nabla u_{h,2}\|_{P_j^2}^2.
			\end{align*}

			\textbf{Step Five, estimate the lower bound of $a_0(u_h, v_h)$:} 
			The full bilinear form $a_0$ now has the following lower bound
			\begin{align*}
				a_0(u_h, v_h) &\ge \left\| \mu_1^{\frac{1}{2}} \nabla u_{h,1} \right\|_{\Omega_1 \setminus P^1}^2 + \left\| \mu_2^{\frac{1}{2}} \nabla u_{h,2} \right\|_{\Omega_2 \setminus P^2}^2 + C_a \sum_{j=1}^{N_p^1} \left\| \mu_1^{\frac{1}{2}} \nabla u_{h,1} \right\|_{P_j^1}^2 \\
				&\quad + C_b \sum_{j=1}^{N_p^2} \left\| \mu_2^{\frac{1}{2}} \nabla u_{h,2} \right\|_{P_j^2}^2 + C_c \sum_{j=1}^{N_p^1}c_0 \left\| \mathfrak{h}^{\frac{-1}{2}} \overline{[\![ u_h ]\!]}^{\Gamma_j^1} \right\|_{\Gamma_j^1}^2,
			\end{align*}
			with the constants
			\begin{align*}
				C_a &= 1 - \epsilon(1+2\omega_1\mu_1), \\
				C_b &= 1 - 2\epsilon\omega_2\mu_2, \\
				C_c &= \alpha \left( \frac{1}{2} - \frac{5C\alpha}{4\epsilon} \right).
			\end{align*}

			Using \Cref{lemma:6.1} it becomes
			\begin{align*}
				a_h(u_h, v_h) &\ge \left\| \mu_1^{\frac{1}{2}} \nabla u_{h,1} \right\|_{\Omega_1 \setminus P^1}^2 + \left\| \mu_2^{\frac{1}{2}} \nabla u_{h,2} \right\|_{\Omega_2 \setminus P^2}^2 + (C_a - w_1 C C_c) \sum_{j=1}^{N_p^1} \left\| \mu_1^{\frac{1}{2}} \nabla u_{h,1} \right\|_{P_j^1}^2 \\
				&\quad + (C_b - w_2 C C_c) \sum_{j=1}^{N_p^2} \left\| \mu_2^{\frac{1}{2}} \nabla u_{h,2} \right\|_{P_j^2}^2 + C_c \sum_{j=1}^{N_p^1} \frac{c_0}{2}\left\| \mathfrak{h}^{\frac{-1}{2}} [\![ u_h ]\!] \right\|_{\Gamma_j^1}^2.
			\end{align*}
			
			We fix $\epsilon = \frac{1}{2(1+2\omega_1\mu_1)}$.  The constant $C_c$ will be positive for $\alpha = \frac{\epsilon}{5C}$. The terms $(C_a - w_1 C C_c)$ and $(C_b - w_2 C C_c)$ will be both positive for
			\begin{align*}
				\alpha < \frac{\min(w_1^{-1},w_2^{-1})}{2C}.
			\end{align*}
			Then we can conclude that there exit a constant $\beta_0$ such that
			\begin{align*}
				\beta_0\interleave u_h\interleave_*^2\le a_0(u_h,v_h)+s_h(u_h,v_h).
			\end{align*}
			\textbf{Step Six, proof of $\interleave v_h\interleave_*\le C\interleave u_h\interleave_*$:} 
			
			With the triangle inequality, it holds
			\begin{align*}
				\interleave v_h \interleave_* \le \interleave u_h \interleave_* + \sum_{j=1}^{N_p^1} \interleave v_{j,1} \interleave_*.
			\end{align*}
			where
			\begin{align*}
				\interleave v_{j,1} \interleave_*^2 = \left\| \mu_1^{\frac{1}{2}} \nabla v_{j,1} \right\|_{P_j^1}^2 + c_0\left\| \mathfrak{h}^{-\frac{1}{2}} v_{j,1} \right\|_{\Gamma_j^1}^2+s_h(v_{j,1},v_{j,1}).
			\end{align*}

			Using the inequality \eqref{brack-est} gives the appropriate upper bound
			\begin{align*}
				&\sum_{j=1}^{N_p}\left(\left\| \mu_1^{\frac{1}{2}} \nabla v_{j,1} \right\|_{P_j^1}^2+s_h(v_{j,1},v_{j,1})\right)\le C\sum_{j=1}^{N_p}\|\mu_1^{\frac{1}{2}}\nabla v_{j,1}\|_{P_j^{1*}}^2\\ 
				&\le C\sum_{j=1}^{N_p}c_0\left\| \mathfrak{h}^{-\frac{1}{2}} \overline{[\![ u_h ]\!]}^{\Gamma_j^1} \right\|_{\Gamma_j^1}^2 \le C \sum_{j=1}^{N_p}c_0\left\| \mathfrak{h}^{-\frac{1}{2}} [\![ u_h ]\!] \right\|_{\Gamma_j^1}^2 \le C \interleave u_h \interleave_*^2,
			\end{align*}
			
			Using the trace inequality and \eqref{patch-est}
			\begin{align*}
				c_0\left\| \mathfrak{h}^{-\frac{1}{2}} v_{j,1} \right\|^2_{\Gamma_j^1} \le C \left\| \mu_1^{\frac{1}{2}} \nabla v_{j,1} \right\|_{P_j^1}^2 \le C\interleave u_h \interleave_*^2.
			\end{align*}
			
			In conclusion, there exists a positive number $\beta$ such that the conclusion holds.
		\end{proof}
	}
	\begin{remark}
		we can define a new $v_{\Gamma,1}=\alpha\sum_{j=1}^{N_p}v_{j,1}$, where $v_{j,1}$ have the property
		\begin{align*}
			|\Gamma_j|^{-1} \int_{\Gamma_j}( -\nabla v_{j,1} )\cdot n \, {\rm d}s = h^{-1} \overline{\![\![ u_h \!]\!]}^{\Gamma_j}.
		\end{align*}
		With the similar proof for \eqref{dis-lbb}, we can obtain \eqref{dis-lbb-adjonit}.
	\end{remark}
	
	\subsection{Regularity for dual problem}
	In this subsection, we prove the regularity of the dual problem that will be used in the next subsection to establish the $L^2$-norm estimate. The proof is similar to that in \Cref{regualar-for-dual}.
	
	\begin{theorem}[Regularity for the dual problem]\label{dual-reg-1-pfree} We define the problem:
		Find $u\in V$ such that
		\begin{align}\label{aux-pfree}
			a^\star_0(u, v)
			=(f, v)
		\end{align}
		holds for all $v\in V$. Then the above problem has a unique solution $u\in V$. In addition, we have
		\begin{align}\label{dual-ref-pfree}
			\mu_1\|u\|_{2,\Omega_1}+\mu_2\|u\|_{2,\Omega_2}\le  C\|f\|_0.
		\end{align}
	\end{theorem}
	\begin{proof}
		The exsitence of a unique solution $u\in V$ followed by LBB condition of $a_0(\cdot,\cdot)$ and the Lax-Milgram Theorem.
		
		We define the  index $i$ as $i=1,2$. Similar to Steps 1 and 2 in the proof of \Cref{dual-reg-1}, we can follow the same steps to obtain the results $-\mu_i\Delta u_i = f|_{\Omega_i}$ and $[\![\mu\nabla u\cdot\bm n]\!]=0$.
		
		Since $H^{\frac1 2}(\Gamma)$ is dense in $H^{-\frac 1 2 }(\Gamma)$, then for any $\theta\in H^{-\frac 1 2}(\Gamma), \|\theta\|_{-\frac 1 2,\Gamma}=1$ and $\varepsilon\in (0,1)$, there exists $\beta_{\theta}\in H^{\frac 1 2}(\Gamma)$ such that
		\begin{align}
			\|\theta-\beta_{\theta}\|_{-\frac 1 2,\Gamma}\le \varepsilon.
		\end{align}
		For any given $\beta_{\theta}\in H^{\frac 1 2}(\Gamma)$, find $v\in V$ such 
		\begin{align*}
			-\nabla\cdot(\mu \nabla v)&=0\qquad \text{in }\Omega,\\
			[\![\mu \nabla v\cdot\bm n]\!]&=0\qquad \text{on }\Gamma,\\
			[\![v]\!]&=\beta_{\theta}\qquad\!\! \text{on }\Gamma.
		\end{align*}
		It holds the stability
		\begin{align}
			\|v_{1}\|_{\Omega_1}+\|v_2\|_{\Omega_2}\le  C\|\beta_\theta\|_{-\frac 1 2,\Gamma}.
		\end{align}\label{beta-thtea-reg2}
		Then we have
		\begin{align}
			\begin{split}
				a^\star_0(u, v)
				&=
				\mu_1\langle u_1,\nabla v_1\cdot\bm n \rangle_{\Gamma}
				-\mu_2\langle  u_2, \nabla v_2\cdot\bm n\rangle_{\Gamma}\\
				&\quad	+\langle \{\!\!\{\mu\nabla u\cdot\bm n\}\!\!\},[\![v]\!] \rangle_{\Gamma}
				-\langle [\![u]\!],\{\!\!\{\mu\nabla v\cdot\bm n\}\!\!\} \rangle_{\Gamma}\\
				&=\langle \{\!\!\{\mu\nabla u\cdot\bm n\}\!\!\},\beta_{\theta} \rangle_{\Gamma}.
			\end{split}
		\end{align}
		So it holds
		\begin{align}
			\langle \{\!\!\{\mu\nabla u\cdot\bm n\}\!\!\},\beta_{\theta} \rangle_{\Gamma}=(f, v_1)_{\Omega_1}+(f, v_2)_{\Omega_2}
		\end{align}
		Then it holds
		{ 
			\begin{align*}
				&\|\{\!\!\{\mu\nabla u\cdot\bm n\}\!\!\}\|_{\frac 1 2,\Gamma}\\
				&\qquad =\sup_{\theta\in H^{-\frac 1 2}(\Gamma),\|\theta\|_{-\frac 1 2,\Gamma}=1}\langle\{\!\!\{\mu\nabla u\cdot\bm n\}\!\!\}, \theta \rangle_{\Gamma}\\
				&\qquad=\sup_{\theta\in H^{-\frac 1 2}(\Gamma),\|\theta\|_{-\frac 1 2,\Gamma}=1}\langle \{\!\!\{\mu\nabla u\cdot\bm n\}\!\!\}, \theta -\beta_{\theta}+\beta_{\theta}\rangle_{\Gamma}\\
				&\qquad\le 	\|\{\!\!\{\mu\nabla u\cdot\bm n\}\!\!\}\|_{\frac 1 2,\Gamma}\varepsilon
				+C(\|f\|_{\Omega_1}
				+\|f\|_{\Omega_2})(1+\varepsilon)
				,
		\end{align*}}
		
		which leads to
		\begin{align}
			\|\{\!\!\{\mu\nabla u\cdot\bm n\}\!\!\}\|_{\frac 1 2,\Gamma}\le C\frac{1+\varepsilon}{1-\varepsilon}(\|f\|_{\Omega_1}
			+\|f\|_{\Omega_2}).
		\end{align}
		So
		\begin{align}
			\|\{\!\!\{\mu\nabla u\cdot\bm n\}\!\!\}\|_{\frac 1 2,\Gamma}\le C(\|f\|_{\Omega_1}
			+\|f\|_{\Omega_2}).
		\end{align}
		Therefore,
		\begin{align*}
			\|
			\mu_1\nabla u_1\cdot\bm n\|_{\frac 1 2,\Gamma}
			&=	\|
			\{\!\!\{\mu\nabla u\cdot\bm n\}\!\!\}+w_1[\![\mu\nabla u\cdot\bm n]\!]\|_{\frac 1 2,\Gamma} \\
			&\le\|\{\!\!\{\mu\nabla u\cdot\bm n\}\!\!\}\|_{\frac 1 2,\Gamma} +\|[\![\mu\nabla u\cdot\bm n]\!]\|_{\frac 1 2,\Gamma} \\
			&=\|\{\!\!\{\mu\nabla u\cdot\bm n\}\!\!\}\|_{\frac 1 2,\Gamma} \\
			&\le  C\|f\|_0.
		\end{align*}
		Then by \eqref{con-infsup-1}, we can get
		\begin{align*}
			\|u_i\|_{1,\Omega_i}\le C\sup_{0\neq v\in V}\frac{a^{\star}_0(u,v)}{\|v\|_1}\le C\frac{(f,v)}{\|v\|_1}\le C\|f\|_0
		\end{align*}
		Finally, it holds
		\begin{align*}
			&\mu_1\| u\|_{2,\Omega_1} + 	\mu_2\| u\|_{2,\Omega_2} \\
			&\qquad\le  C(\mu_1\|\Delta u_1\|_{\Omega_1}
			+\mu_2\|\Delta u_2\|_{\Omega_2} \\
			&\qquad\quad
			+\|
			\mu_1\nabla u_1\cdot\bm n+[\![u]\!]\|_{\frac 1 2,\Gamma}
			+\|[\![\mu\nabla u\cdot\bm n]\!]\|_{\frac 1 2,\Gamma})\\
			&\qquad\le  C(\mu_1\|\Delta u_1\|_{\Omega_1}
			+\mu_2\|\Delta u_2\|_{\Omega_2} \\
			&\qquad\quad
			+\|
			\mu_1\nabla u_1\cdot\bm n\|_{\frac 1 2,\Gamma}+\sum_{i=1}^2\|u_i\|_{1,\Omega_i}
			+\|[\![\mu\nabla u\cdot\bm n]\!]\|_{\frac 1 2,\Gamma})\\
			&\qquad\le C\|f\|_0.
		\end{align*}
		We have finished  our proof.
	\end{proof}
	\subsection{Error estimations}
	{ 
		\begin{theorem}[Energy error estimation]\label{penaltyfree-energy}
			Let $u\in V$ and $u_h\in V_h$ be the solution of \eqref{fem-org-penaltyfree} and \eqref{fem-penalty-free}, respectively. Then for $u_1\in H^{k+1}(\Omega)$ and $u_2\in H^{k+1}(\Omega)$, we have
			\begin{align*}
				\interleave u-u_h\interleave_*
				\le Ch^k(\mu_1^{\frac 1 2}|u_1|_{k+1,\Omega_1}+\mu_2^{\frac 1 2}|u_2|_{k+1,\Omega_2}).
			\end{align*}
		\end{theorem}
		
		\begin{proof} 
			Due to $s_h(u,v_h)=0$, then
			\begin{align*}
				s_h(\mathcal{J}_hu,v_h)&\le s_h(\mathcal{J}_hu-u,\mathcal{J}_hu-u)^{\frac{1}{2}}s_h(v_h,v_h)^{\frac{1}{2}}\\
				&\le C(\sum_{i=1}^2\mu_i\sum_{T\in\Omega_i^*}\sum_{l=k}^{k+1}h^{2(l-1)}\|D^l(u_i-\mathcal{J}_hu_i)\|^2_T)^{\frac{1}{2}}\interleave v_h\interleave_*\\
				&\le  Ch^k(\mu_1^{\frac{1}{2}}|u_1|_{k+1,\Omega_1}+\mu_2^{\frac{1}{2}}|u_2|_{k+1,\Omega_2})\interleave v_h\interleave_*.
			\end{align*}
			By \Cref{thm:dis-lbb2}, we have
			\begin{align*}
				\interleave u_h-\mathcal J_h u\interleave_*
				&\le \sup_{0\neq v_h\in V_h}\frac{a(u_h- \mathcal J_h u,  v_h)+s_h(\mathcal{J}_hu-u_h,v_h)}{\beta\interleave v_h\interleave_*} \\
				&=  \sup_{0\neq v_h\in V_h}\frac{a(  u- \mathcal J_h u,  v_h)-s_h(\mathcal{J}_hu,v_h)}{\beta\interleave v_h\interleave_*} \\
				&\le C(\frac{M}{\beta}\| u-\mathcal J_h u\|_*+s_h(u- J_h u,u- J_h u)^{\frac{1}{2}}).
			\end{align*}
			With the triangle inequality to
			\begin{align*}
				\interleave u-u_h\interleave_*&\le \interleave u-\mathcal{J}_hu\interleave_*+\interleave u_h-\mathcal{J}_h u\interleave_*\\
				&\le C(\frac{M}{\beta}+1)\|u-\mathcal{J}_hu\|_*+s_h(u- \mathcal J_h u,u- \mathcal J_h u)^{\frac{1}{2}})\\
				&\le Ch^k(\mu_1^{\frac 1 2}|u_1|_{k+1,\Omega_1}+\mu_2^{\frac 1 2}|u_2|_{k+1,\Omega_2})
			\end{align*}.
		\end{proof}
	}
	Before we give the theorem for the $L^2$ error estimation, we first give a estimation for $a_0(\cdot,\cdot)$.
	\begin{lemma}\label{lem:dual-penalty-free} Let $u\in V$ and $u_h\in V_h$ be the solution of \eqref{fem-org-penaltyfree} and \eqref{fem-penalty-free}, respectively. Then for $u_1\in H^{k+1}(\Omega_1)$, $u_2\in H^{k+1}(\Omega_2)$, $\Phi\in V$, $\Phi_1\in H^2(\Omega_1)$
		and $\Phi_2\in H^2(\Omega_2)$, we have
		{ 
			\begin{align*}
				&|a_0(u-u_h,\Phi-\mathcal J_h\Phi)|
				\\&\le C          h^{k+1}
				(\mu_1^{\frac 1 2}|u_1|_{k+1,\Omega_1}+\mu_2^{\frac 1 2}|u_2|_{k+1,\Omega_2}) (\mu_1^{\frac 1 2}|\Phi_1|_{2,\Omega_1}+
				\mu_2^{\frac 1 2}|\Phi_2|_{2,\Omega_2}
				) .
		\end{align*}}
	\end{lemma}
	\begin{proof} We use the definition of $a(\cdot,\cdot)$ to get
		\begin{align*}
			&a_0(u-u_h,\Phi-\mathcal J_h\Phi)\\
			&\quad=\mu_1(\nabla (u_1-u_{h,1}), \nabla (\Phi_1-\mathcal J_h\Phi_1) )_{\Omega_1}
			+ \mu_2(\nabla (u_2-u_{h,2}), \nabla (\Phi_2-\mathcal J_h\Phi_2))_{\Omega_2}\\
			&\qquad
			-\langle \{\!\!\{\mu\nabla (u-u_h)\cdot\bm n\}\!\!\},[\![\Phi-\mathcal J_h \Phi]\!] \rangle_{\Gamma}
			+\langle [\![u-u_h]\!],\{\!\!\{\mu\nabla (\Phi-\mathcal J_h\Phi)\cdot\bm n\}\!\!\} \rangle_{\Gamma}\\
			&\quad=:\sum_{i=1}^4 E_i
		\end{align*}
		Similar to proof of \Cref{lem:dual}, we can get estimates
		
		\begin{align*}
			\begin{split}
				|E_1+E_2|
				&\le C          h^{k+1}
				(\mu_1^{\frac 1 2}|u_1|_{k+1,\Omega_1}+\mu_2^{\frac 1 2}|u_2|_{k+1,\Omega_2}) (\mu_1^{\frac 1 2}|\Phi_1|_{2,\Omega_1}+
				\mu_2^{\frac 1 2}|\Phi_2|_{2,\Omega_2}
				),
			\end{split}
		\end{align*}
		and
		\begin{align*}
			\begin{split}
				|E_3|
				&\le C          h^{k+1}
				(\mu_1^{\frac 1 2}|u_1|_{k+1,\Omega_1}+\mu_2^{\frac 1 2}|u_2|_{k+1,\Omega_2}) (\mu_1^{\frac 1 2}|\Phi_1|_{2,\Omega_1}+
				\mu_2^{\frac 1 2}|\Phi_2|_{2,\Omega_2}
				) ,
			\end{split}
		\end{align*}
		and
		\begin{align*}
			\begin{split}
				|E_4|
				&\le C          h^{k+1}
				(\mu_1^{\frac 1 2}|u_1|_{k+1,\Omega_1}+\mu_2^{\frac 1 2}|u_2|_{k+1,\Omega_2}) (\mu_1^{\frac 1 2}|\Phi_1|_{2,\Omega_1}+
				\mu_2^{\frac 1 2}|\Phi_2|_{2,\Omega_2}
				) .
			\end{split}
		\end{align*}
		We use all the estimations above to get finish our proof.
	\end{proof}	
	
	
	{ 
		\begin{theorem}[$L^2$ error estimation]\label{L2-penalty-free}  Let $u\in V$ and $u_h\in V_h$ be the solution of \eqref{fem-org-penaltyfree} and \eqref{fem-penalty-free}, respectively. Then for $u_1\in H^{k+1}(\Omega)$ and $u_2\in H^{k+1}(\Omega)$, we have
			\begin{align*}
				\| u-u_h\|_0
				\le Ch^{k+1}\max(\mu_1^{-\frac 1 2},\mu_1^{-\frac{3}{2}},\mu_2^{-\frac 1 2},\mu_2^{-\frac{3}{2}})
				(\mu_1^{\frac 1 2}|u_1|_{k+1,\Omega_1}+\mu_2^{\frac 1 2}|u_2|_{k+1,\Omega_2}).
			\end{align*}
		\end{theorem}
		\begin{proof} We define the dual problem: Find $\Phi\in V$ such that
			\begin{align}\label{dual2}
				a^\star_0(\Phi, v)
				=(u-u_h, v)
			\end{align}
			holds for all $v\in V$. Then by \Cref{dual-reg-1}, the above problem has a unique solution $u\in V$ satisfying
			\begin{align}\label{dualreg2}
				\mu_1\|\Phi_1\|_{2,\Omega_1}+\mu_2\|\Phi_2\|_{2,\Omega_2}\le  C\max\{\mu_1^{-1},\mu_2^{-1},1\}\|u-u_h\|_0.
			\end{align}
			We take $v=u-u_h$ in \eqref{dual2} ,  the property \eqref{sys} and  the orthogonality in \eqref{a_0or} to get
			\begin{align*}
				\|u-u_h\|_0^2 =	a^\star(\Phi, u-u_h) 
				=a(u-u_h,\Phi)
				=a(u-u_h,\Phi-\mathcal J_h\Phi)+s_h(u_h,\mathcal{J}_h\Phi).
			\end{align*}
			With the result we got in \Cref{penaltyfree-energy}
			\begin{align*}
				s_h(u_h,\mathcal{J}_h\Phi)\le C\interleave u-u_h\interleave_*\interleave\Phi-\mathcal{J}_h\Phi\interleave_*,
			\end{align*}
			and the result in \Cref{lem:dual}, the regularity estimation \eqref{dualreg2} to get
			\begin{align*} 
				&\|u-u_h\|_0^2            \\
				&\le         C h^{k+1}
				(\mu_1^{\frac 1 2}|u_1|_{k+1,\Omega_1}+\mu_2^{\frac 1 2}|u_2|_{k+1,\Omega_2}) (\mu_1^{\frac 1 2}|\Phi_1|_{2,\Omega_1}+
				\mu_2^{\frac 1 2}|\Phi_2|_{2,\Omega_2}
				)\\
				&
				\le         C h^{k+1}
				\max(\mu_1^{-\frac 1 2},\mu_2^{-\frac 1 2})
				(\mu_1^{\frac 1 2}|u_1|_{k+1,\Omega_1}+\mu_2^{\frac 1 2}|u_2|_{k+1,\Omega_2}) (\mu_1\|\Phi_1\|_{2,\Omega_1}+
				\mu_\|\Phi_2\|_{2,\Omega_2}
				)\\
				&\le C          h^{k+1}
				\max(\mu_1^{-\frac 1 2},\mu_1^{-\frac{3}{2}},\mu_2^{-\frac 1 2},\mu_2^{-\frac{3}{2}})
				(\mu_1^{\frac 1 2}|u_1|_{k+1,\Omega_1}+\mu_2^{\frac 1 2}|u_2|_{k+1,\Omega_2}) \|u-u_h\|_0.
			\end{align*}
		\end{proof}
	}

	\bibliographystyle{plain}
	\bibliography{references}

\end{document}